\documentclass[11pt]{amsart}

\usepackage{euscript}
\usepackage{amsmath}
\usepackage{amsfonts}
\usepackage{amssymb}
\usepackage{mathrsfs}
\usepackage{amsthm}
\usepackage{epsfig}
\usepackage{epstopdf}
\usepackage{url}
\usepackage{array}
\usepackage{amssymb}\usepackage{amscd}
\usepackage{epic}
\usepackage{eufrak}
\usepackage{tikz,underscore}
\usepackage{tikz-cd}
\usepackage{float}
\usepackage{subfigure}
\usepackage{tkz-euclide}

\usepackage{graphicx}
\usetkzobj{all}

\numberwithin{equation}{section}

\theoremstyle{plain}
\newtheorem{theorem}{Theorem}[section]
\newtheorem{lemma}[theorem]{Lemma}
\newtheorem{proposition}[theorem]{Proposition}

\newtheorem{corollary}[theorem]{Corollary}

\newtheorem{thm}[equation]{Theorem}
\newtheorem{cor}[equation]{Corollary}

\newtheorem*{claim*}{Claim}

\newtheorem{ques}[theorem]{Question}

\theoremstyle{definition}

\newtheorem{remark}[theorem]{Remark}

\newtheorem{definition}[theorem]{Definition}

\DeclareMathOperator{\Isom}{{\mathrm Isom}}
\DeclareMathOperator{\Hull}{{\mathrm Hull}}
\DeclareMathOperator{\stab}{{\mathrm stab}}
\DeclareMathOperator{\hd}{{\mathrm hd}}

\def\La{\Lambda}
\def\la{\lambda}

\def\Ga{\Gamma}
\def\ga{\gamma}

\def\H{\mathbb H}

\def\R{\mathbb R}
\def\Z{\mathbb Z}

\def\N{\mathbb N}

\newcommand{\calF}{{\mathcal F}}

\newcommand{\Core}{\textup{Core}}

\newcommand{\thin}{\textup{thin}} 
\newcommand{\noncusp}{\textup{noncusp}}

\newcommand{\card}{{\mathrm{card}}\, }
\def\geo{\partial_{\infty}}

\newcommand{\Bis}{\textup{Bis}}

\title{Geometric finiteness in negatively pinched Hadamard manifolds}

\author{Michael Kapovich}
\address{M.K.: Department of Mathematics, UC Davis, One Shields Avenue, Davis CA 95616, USA}
\address{KIAS, 85 Hoegiro, Dongdaemun-gu, 
Seoul 130-722, South Korea}

\email{kapovich@math.ucdavis.edu}

\author{Beibei Liu}
\address{B.L.: Department of Mathematics, UC Davis, One Shields Avenue, Davis CA 95616, USA}
\email{bxliu@math.ucdavis.edu}


\providecommand{\keywords}[1]{\textbf{\textit{Index terms---}} #1}

\date{December 9, 2018}

\begin{document}

\begin{abstract}
In this paper, we generalize Bonahon's characterization of geometrically infinite torsion-free discrete subgroups of $\textup{PSL}(2, \mathbb{C})$ to geometrically infinite discrete  subgroups $\Gamma$ of isometries of negatively pinched Hadamard manifolds $X$. We then  generalize a theorem of Bishop to prove that every discrete geometrically infinite isometry subgroup $\Gamma$ has a set of nonconical limit points with the cardinality of the continuum.  
\end{abstract}


\maketitle

\section{Introduction}
\label{sec:introduction}

The notion of geometrically finite  discrete groups  was originally introduced by Ahlfors in \cite{Ah}, for subgroups of isometries of 
the $3$-dimensional hyperbolic space $\H^{3}$ as the finiteness condition for the number of faces of a convex fundamental polyhedron. In the same paper, Ahlfors proved that the limit set of a geometrically finite subgroup of isometries of $\H^3$ has either zero or full Lebesgue measure in $S^2$. 
The notion of geometric finiteness turned out to be quite fruitful in the study of Kleinian groups. Alternative definitions of geometric finiteness were later given by Marden \cite{Ma}, Beardon and Maskit \cite{BM}, and Thurston \cite{Th}. These definitions were further extended by Bowditch \cite{Bo1} and Ratcliffe \cite{Rat} for isometry subgroups of higher dimensional hyperbolic spaces and, a bit later, by Bowditch \cite{Bo2} to negatively pinched Hadamard manifolds. While the original Ahlfors' definition turned out to be too limited (when used beyond the hyperbolic 3-space), other definitions of geometric finiteness were proven to be equivalent by Bowditch in \cite{Bo2}. 

Our work is motivated by the definition of geometric finiteness due to Beardon and Maskit \cite{BM} who proved 

\begin{theorem}
A discrete isometry subgroup $\Gamma$ of $\H^3$ is geometrically finite if and only if every limit point of $\Gamma$  is either a {\em conical limit point} or  a {\em bounded parabolic fixed point}. 
\end{theorem}

This theorem was improved by Bishop in \cite{Bi}:

 \begin{theorem}
 \label{theo 1.2}
A discrete subgroup $\Gamma< \Isom(\H^3)$ is geometrically finite if and only if  every point of $\Lambda( \Gamma)$ is either a conical limit point or a parabolic fixed point. Furthermore, if $\Gamma< \Isom(\H^3)$ is geometrically infinite, $\Lambda( \Gamma)$ contains a set  of nonconical limit points with the cardinality of the continuum. 
 \end{theorem}

The key ingredient in Bishop's proof of Theorem \ref{theo 1.2} is Bonahon's theorem\footnote{Bonahon uses this result to prove his famous theorem about tameness of hyperbolic 3-manifolds. } \cite{Bo}:

\begin{theorem} 
\label{Bona}
A discrete torsion-free   subgroup $\Gamma< \Isom(\H^3)$ is geometrically infinite if and only if  there exists a sequence of closed geodesics 
$\lambda_i$ in the manifold $M=\H^3/\Gamma$ which ``escapes every compact subset of $M$,'' i.e., for every compact subset $K\subset M$,
$$
\card (\{i: \lambda_i\cap K\ne \emptyset\})<\infty. 
$$
\end{theorem}

According to Bishop, Bonahon's theorem also holds for groups with torsion. We extend Bonahon's proof and prove that Bonahon's theorem  holds for discrete isometry subgroups of  negatively pinched Hadamard manifolds $X$.

\medskip 
Bowditch generalized the notion of geometric finiteness to discrete subgroups of isometries of negatively pinched Hadamard manifolds \cite{Bo2}. A negatively pinched Hadamard manifold is a complete, simply connected Riemannian manifold such that all sectional curvatures lie between two negative constants. From now on, we use $X$ to denote an $n$-dimensional  negatively pinched Hadamard manifold, $\geo X$ its visual (ideal) boundary, $\bar{X}$ the visual compactification $X\cup \geo X$, $\Gamma$ a discrete subgroup of isometries of $X$, $\Lambda=\Lambda(\Gamma)$ the limit set of $\Gamma$. The convex core $\Core(M)$ of $M=X/\Gamma$ is defined as the $\Gamma$-quotient of the closed convex hull of $\Lambda (\Gamma)$ in $X$. Recall also that  a point $\xi\in \geo X$ is a \emph{conical limit point}\footnote{Another way  is to describe conical limit points of $\Gamma$ as points $\xi\in \geo X$ such that one, equivalently, every, geodesic ray $ \R_+\to X$ asymptotic to $\xi$ projects to a non-proper map $\R_+\to M$.} of $\Gamma$ if for every $x\in X$ and every geodesic ray $l$ in $X$ asymptotic to $\xi$, there exists a positive constant $A$ such that the set $\Gamma x\cap N_{A}(l)$ accumulates to $\xi$, where    $N_{A}(l)$ denotes the $A$-neighborhood of $l$ in $X$. A  parabolic fixed point $\xi\in \geo X$ (i.e.  a fixed point of a parabolic element of $\Gamma$) is called {\em bounded} if 
$$
(\Lambda(\Gamma)- \{ \xi \})/ {\Gamma}_\xi$$
 is compact. Here ${\Gamma}_\xi$ is the stabilizer of $\xi$ in $\Gamma$.

 Bowditch  \cite{Bo2}, gave four equivalent definitions of geometric finiteness for $\Gamma$:

\begin{theorem}\label{thm:gfcharacterization}

The followings are equivalent for discrete subgroups $\Gamma< \Isom(X)$:
\begin{enumerate}
\item
The quotient space $\bar{M}(\Gamma)=(\bar{X} - \Lambda)/ \Gamma$ has finitely many topological ends each of which is a ``cusp''.

\item The limit set $\Lambda(\Gamma)$ of $\Gamma$ consists entirely of conical limit points and bounded parabolic fixed points. 

\item The noncuspidal part of the convex core $\Core(M)$ of $M=X/ \Gamma$ is compact. 

\item
For some $\delta >0$, the uniform $\delta$-neighbourhood of the convex core, $N_\delta(\Core(M))$, has finite volume and there is a bound on the orders of finite subgroups of $\Gamma$. 
\end{enumerate}
\end{theorem}

If one of these equivalent conditions holds, the subgroup $\Gamma<\Isom(X)$ is said to be geometrically finite; otherwise, $\Gamma$ is said to be geometrically infinite. 

The main results of our paper are:

\begin{theorem}
\label{theo 1.3}
Suppose that $\Gamma< \Isom(X)$ is a  discrete subgroup. Then the followings are equivalent:

\begin{enumerate}
\item $\Gamma$ is geometrically infinite.

\item There exists a sequence of closed geodesics $\lambda_i\subset M=X/\Gamma$ which escapes every compact subset of $M$. 

\item The set of nonconical limit points of $\Gamma$ has the cardinality of the continuum. 

\end{enumerate}
\end{theorem}

\begin{corollary}
\label{coro 1.6}
If $\Gamma<  \Isom(X)$  is a  discrete subgroup then $\Gamma$ is geometrically finite if and only if every limit point of $\Gamma$ 
is either a conical limit point or a parabolic fixed point. 
\end{corollary}


These results can be sharpened as follows. We refer the reader to section \ref{sec:ends} for the precise definitions of ends $e$ 
of the orbifolds $Y=\Core(M)$ and $\noncusp_{\varepsilon}(Y)$, of their {\em neighborhoods} $C\subset Y$  
and of their {\em end-limit sets} $\Lambda(C)$, $\Lambda(e)$, which are certain subsets   
of the set of non-conical limit points of $\Gamma$.

In \cite[Section 4]{FMS}, Falk, Matsuzaki and Stratmann conjectured the end-limit set of an end $e$ of $Y$ 
is countable if and only if $\Lambda(e)$ is the $\Gamma$-orbit of a (bounded) parabolic fixed point. A slight modification of 
the proof of Theorem \ref{theo 1.3} proves this conjecture:  

\begin{corollary}
\label{cusp end}
$\Lambda(e)$  is countable if and only if the end $e$ of $Y$ is a cusp. 
\end{corollary}

Furthermore: 

\begin{corollary}\label{cor:ends} 
Let $C\subset Y$ be an unbounded complementary component of a compact subset $K\subset Y$. 
The  $\Lambda(C)$  is countable if and only if $C$ is Hausdorff-close to a finite union of cuspidal neighborhoods of cusps 
in $Y$. 
\end{corollary}

\medskip 
By Theorem \ref{theo 1.3}, the set of nonconical limit points of a  discrete isometry subgroup $\Gamma$ has the cardinality of the continuum if $\Gamma$ is geometrically infinite. It is natural to ask: 

\begin{ques}
What is the  Hausdorff dimension of the set of nonconical limit points of $\Gamma$? Here, the Hausdorff dimension is defined with respect to any of the {\em visual metrics} on $\geo X$, see \cite{Pau}. 
\end{ques}

Partial results have been obtained  by Fern\'andez and Meli\'an \cite{FM} in the case of 
Fuchsian subgroups of the 1st kind, $\Gamma< \Isom(\H^2)$ and by Bishop and Jones \cite{BJ} in the case of finitely generated discrete torsion-free subgroups $\Ga< \Isom(\H^3)$ of the 2nd kind, such that 
the manifold $\H^3/\Gamma$ has injectivity radius bounded below. In both cases, the Hausdorff dimension of the set of nonconical limit points equals the Hausdorff dimension of 
the entire limit set. 

\medskip 
Below is an outline of the proof of Theorem  \ref{theo 1.3}. Our proof of the implication (1)$\Rightarrow$(2) mostly follows Bonahon's argument with the following exception: At some point of the proof Bonahon has to show that certain elements of $\Gamma$ are loxodromic. For this he uses a calculation with $2\times 2$ parabolic matrices: If $g, h$ are parabolic elements of $\Isom(\H^3)$ generating a nonelementary subgroup then either $gh$ or $hg$ is non-parabolic. This argument is no longer valid for isometries of higher dimensional hyperbolic spaces, let alone Hadamard manifolds. We replace this computation with a more difficult argument showing that there exists a number $\ell=\ell(n,\kappa)$ such that for every $n$-dimensional Hadamard manifold $X$ with sectional curvatures pinched between $-\kappa^2$ and $-1$ and for any pair of parabolic isometries $g, h\in \Isom(X)$ generating a nonelementary discrete subgroup,  a certain word $w=w(g,h)$ of length $\le \ell$ is loxodromic (Theorem \ref{proposition 3.16}).  We later found a stronger result by Breuillard and Fujiwara \cite{BF} that there exists a loxodromic element of uniformly bounded word length in the discrete nonelementary subgroup generated by any isometry subset (Corollary \ref{cor:BF}), which can be used to deal with nonelementary groups generated by elliptic isometries. 

Our proof of the implication (2)$\Rightarrow$(3) is similar to Bishop's but is more coarse-geometric in nature. Given a sequence of closed geodesics $\lambda_i$ in $M$ escaping compact subsets, we define a family of proper piecewise geodesic paths $\gamma_{\tau}$ in $M$  consisting of alternating geodesic arcs $\mu_i, \nu_i$, such that $\mu_i$ connects $\lambda_i$ to $\lambda_{i+1}$ and is orthogonal to both, while the image of $\nu_i$ is contained in the loop $\lambda_{i}$. If the lengths of $\nu_i$ are sufficiently long, then the path $\gamma_{\tau}$ lifts to a uniform quasigeodesic $\tilde{\gamma}_{\tau}$ in $X$, which, therefore, is uniformly close to a geodesic $\tilde{\gamma}^*_{\tau}$. Projecting the latter to $M$, we obtain a geodesic $\gamma^*_{\tau}$ uniformly close to $\gamma_{\tau}$, which  implies that the ideal point 
$\tilde{\gamma}^*_{\tau}(\infty)\in \geo X$ is a nonconical limit point of $\Gamma$. Different choices of the arcs $\nu_i$ yield distinct limit points, which, in turn implies that $\Lambda(\Gamma)$ contains a set of nonconical limit points with the cardinality of the continuum. 
The direction (3)$\Rightarrow$(1) is a direct corollary of Theorem \ref{thm:gfcharacterization}.

\medskip 
{\bf Organization of the paper.} In Section \ref{sec:review}, we review the angle comparison theorem \cite[Proposition 1.1.2]{Bo2} for negatively pinched Hadamard manifolds and derive some useful geometric inequalities. In Section \ref{sec:elementary}, we review the notions of elementary  subgroups of isometries  of negatively pinched Hadamard manifolds, \cite{Bo2}. In Section \ref{sec:thick-thin}, we review  the thick-thin decomposition in negatively pinched Hadamard manifolds and some properties of parabolic subgroups, \cite{Bo2}. In Section \ref{sec:quasigeodesics}, we use the results in Section \ref{sec:review} to prove that certain piecewise geodesic paths in Hadamard manifolds with sectional curvatures $\leq -1$ are uniform quasigeodesics. In Section \ref{sec:loxodromic}, we explain how to  produce loxodromic isometries as words $w(g,h)$ of  uniformly bounded  length, where $g, h$ are 
parabolic isometries of $X$  with distinct fixed points.  In Section \ref{sec:Bohanon}, we generalize Bonahon's theorem, the implication (1)$\Rightarrow$(2) in Theorem \ref{theo 1.3}. In Section \ref{sec:continuum}, we  construct  the set of nonconical limit points with the cardinality of the continuum and complete the proof of Theorem \ref{theo 1.3}. Lastly, in Section \ref{sec:ends} we prove Corollaries \ref{cusp end} and \ref{cor:ends}. \\

{\bf Acknowledgements.} We would like to thank the referees for the useful comments and suggestions. The first author was partly supported by the NSF grant  DMS-16-04241 as well as by 
KIAS (the Korea Institute for Advanced Study) through the KIAS scholar program. 
Some of this work was done during his stay at KIAS and he is thankful to KIAS for its hospitality. 
The second author was partly supported by the NSF grant DMS-17-00814.

\section{Notation}\label{sec:notation}

In a metric space $(Y,d)$,  we will use the notation $B(a,r)$ to denote the {\em open $r$-ball} centered at $a$ in $Y$. For a 
 subset $A\subset Y$ and a point $y\in Y$, we will denote by $d(y,A)$ the {\em minimal distance} from $y$ to $A$, i.e. 
$$
d(y,A):= \inf \{d(y,a) \mid  a\in A\}.  
$$
Similarly, for two subsets $A, B\subset Y$ define their minimal distance as
$$
d(A,B)= \inf \{d(a,b) \mid  a\in A, ~~b\in B\}. 
$$
We will use the notation $l(p)$ or $length(p)$ for the length of  a rectifiable path $p$ in a metric space. 

We use the notation $\bar{N}_r(A)$ for the {\em closed} $r$-neighborhood of $A$ in $Y$:
$$
\bar{N}_r(A)= \{y\in Y: d(y,A)\le r\}. 
$$

The \emph{Hausdorff distance} $\hd(Q_1, Q_2)$ between two closed subsets $Q_{1}$ and $Q_{2}$ of $(Y,d)$ 
is the infimum of $r\in [0, \infty)$ such that  $Q_{1}\subseteq \bar{N}_{r}(Q_{2})$ and $Q_{2}\subseteq \bar{N}_{r}(Q_{1})$.

\medskip 
Throughout the paper, $X$ will denote an $n$-dimensional negatively pinched Hadamard manifold, unless otherwise stated; we assume that all sectional curvatures of $X$ lie between $-\kappa^{2}$ and $-1$, where $\kappa>0$. Note that the lower bound $-\kappa^{2}$ is used essentially in a property of quasiconvex subsets (Proposition \ref{convex}), 
 Margulis Lemma  (Section \ref{sec:thick-thin}), in Sections \ref{sec:loxodromic},  \ref{sec:generalization} and  \ref{sec:ends}.   
We let $d$ denote the Riemannian distance function on $X$ and let $\Isom(X)$ denote the isometry group of $X$. 

For a Hadamard manifold $X$, the exponential map is a diffeomorphism, in particular, 
$X$ is diffeomorphic to $\R^{n}$.
Then $X$ can be compactified by adjoining the ideal boundary sphere $\geo X$, and we will use the notation $\bar{X}=X\cup \geo X$ for this compactification. The space $\bar{X}$ is homeomorphic to the closed $n$-dimensional  ball.

In this paper, geodesics will be always parameterized by their arc-length; we will conflate geodesics in $X$ with their images.

\medskip
Given a closed subset $A\subseteq X$ and $x\in X$, we write
$$\textup{Proj}_{A}(x)=\lbrace y\in A \mid d(x, y)=d(x, A) \rbrace $$
for the nearest-point projection of $x$ to $A$. It consists of all points in $A$ which are closest to $x$. If $A$ is convex, then $\textup{Proj}_{A}(x)$ is a singleton. 

Hadamard spaces are uniquely geodesic and we will let $xy\subset X$ denote the geodesic segment  connecting $x\in X$ to $y\in X$. Similarly, given $x\in X$ and $\xi\in \geo X$ we will use the notation $x\xi$ for the unique geodesic ray emanating from $x$ and asymptotic to $\xi$; for two distinct points $\xi, \eta\in \geo X$, we use the notation $\xi\eta$ to denote the unique (up to reparameterization) geodesic asymptotic to $\xi$ and $\eta$.  

Given $\xi\in \geo X$, horospheres about $\xi$ are level sets of a Busemann function $h$ about $\xi$. For details of Busemann functions, see \cite{Ballmann, Bo2} (notice that Bowditch uses a nonstandard notation for Busemann functions, which are negatives of the standard Busemann functions). A set of the form $h^{-1}((-\infty, r])$ for $r\in \mathbb{R}$ is called a {\em horoball} about $\xi$. Horoballs are convex.

Given points $P_{1}, P_{2}, \cdots, P_{m}\in X$ we let $[P_{1}P_{2}\cdots P_{m}]$ denote the geodesic polygon in $X$ which is the union of geodesic segments $P_{i}P_{i+1}$, $i$ taken modulo $m$.

Given two distinct points $x, y\in X$, and a point $q\in xy$, we define   the {\em normal hypersurface} $\mathscr{N}_q(x,y)$, i.e. 
the image of the normal exponential map to the segment $xy$ at the point $q$:
$$ \mathscr{N}_q(x,y)= \exp_q(T_q^{\perp}(xy)),$$
where $T_q^{\perp}(xy)\subset T_qX$ is the orthogonal complement in the tangent space at $q$ to the segment $xy$. In the special case when $q$ is the midpoint of $xy$, $\mathscr{N}_q(x,y)$ is the {\em perpendicular bisector} of the segment $xy$, and we will denote it $\Bis(x,y)$. Similarly, we define the normal hypersurface $\mathscr{N}_{q}(\xi, \eta)$ for any point $q$ in the biinfinite geodesic $\xi\eta$. 

Note that if $X$ is a real-hyperbolic space, then $\Bis(x,y)$ is totally geodesic and equals the set of points equidistant from $x$ and $y$. For general Hadamard spaces, this is not the case. However, if $X$ is $\delta$-hyperbolic, then each 
$\mathscr{N}_p(x,y)$ is $\delta$-quasiconvex, see Definition \ref{def:qc}.

We let $\delta$ denote the {\em hyperbolicity constant} of $X$; hence, $\delta\le \cosh^{-1}(\sqrt{2})$. We will use the notation $\Hull(A)$
for the  {\em closed convex hull} of a subset  $A\subset  X$, i.e. the intersection of all closed convex subsets of $X$ containing $A$. The notion of the closed  convex hull extends to the closed subsets of $\geo X$ as follows. Given a closed  subset  $A\subset \geo X$, we denote by  $\Hull(A)$  the smallest closed convex subset of $X$ whose
accumulation set in $\bar{X}$ equals $A$. (Note that  $\Hull(A)$ is nonempty as long as $A$ contains more than one point.)

For a subset $A\subset X$ the {\em quasiconvex hull}  $\textup{QHull}(A)$ of $A$ in $X$ is defined as the union of all geodesics connecting points of $A$. Similarly, for a closed subset 
$A\subset \geo X$, the quasiconvex hull  $\textup{QHull}(A)$ is  the union of all biinfinite geodesics asymptotic to points of $A$. 
Then $\textup{QHull}(A)\subset  \Hull(A)$.

\medskip 
We will use the notation 
$\Gamma$ for a discrete  subgroup of isometries of $X$. We let $\Lambda=\Lambda(\Gamma)\subset \geo X$ denote the {\em  limit set} of $\Gamma$, i.e. the accumulation set in $\geo X$  
of one (equivalently, any) $\Ga$-orbit in $X$. The group  $\Gamma$ acts properly discontinuously on $\bar{X}\setminus \Lambda$, \cite[Proposition 3.2.6]{Bo2}. We obtain an orbifold with boundary
$$
\bar{M}= \left(\bar{X}\setminus \Lambda \right)/ \Gamma .$$
If $\Gamma$ is torsion-free, then $\bar{M}$ is a partial compactification of the quotient manifold $M=X/ \Gamma$. We let $\pi: X\rightarrow M$ denote the covering projection.

\section{Review of negatively pinched Hadamard manifolds}
\label{sec:review}

\subsection{Metric comparison inequalities} 

For any triangle $[ABC]$ in $(X, d)$, we define a comparison triangle $[A'B'C']$ for $[ABC]$ in $(\mathbb{H}^{2}, d')$ as follows.

\begin{definition}
For a triangle  $[ABC]$ in $(X, d)$, let $A', B', C'$ be $3$ points in the hyperbolic plane  $(\H^{2}, d')$ satisfying that $d'(A', B')=d(A, B), d'(B', C')=d(B, C)$ and $d'(C', A')=d(C, A)$. Then $[A'B'C']$ is called a {\em comparison triangle} for $[ABC]$. 

\end{definition}

In general, for any geodesic polygon $[P_{1}P_{2}\cdots P_{m}]$ in $(X, d)$, we define a comparison polygon $[P'_{1}P'_{2}\cdots P'_{m}]$ for $[P_{1}\cdots P_{m}]$ in $(\mathbb{H}^{2}, d')$.

\begin{definition}
For any geodesic polygon $[P_{1}P_{2}\cdots P_{m}]$ in $X$, we pick points $P'_{1},  \cdots, P'_{m}$ in  $\mathbb{H}^{2}$ such that $[P'_{1}P'_{i}P'_{i+1}]$ is a comparison triangle for $[P_{1}P_{i}P_{i+1}]$ and the triangles  $[P'_{1}P'_{i-1}P'_{i}]$ and $[P'_{1}P'_{i}P'_{i+1}]$ lie on different sides of $P'_{1}P'_{i}$  for each $2\leq i \leq m-1$. The geodesic polygon $[P'_{1}P'_{2}\cdots P'_{m}]$ is called a \emph{comparison polygon} for $[P_{1}P_{2}\cdots P_{m}]$. 

\end{definition}

\begin{remark}
Such a comparison polygon $[P'_{1}P'_{2}\cdots P'_{m}]$  is not necessarily convex and embedded\footnote{I.e. the natural map $S^1\to \H^2$ defined by tracing the oriented edges of the polygon in the cyclic order need not be injective.}. In the rest of the section, we have additional assumptions for the polygons $[P_{1}P_{2}\cdots P_{m}]$.  Under these assumptions,  their comparison polygons in $\mathbb{H}^{2}$ are embedded and convex, see Corollary \ref{lemma 2.5}. 

\end{remark}

One important property of negatively pinched Hadamard manifolds $X$ is the following 
{\em angle comparison} theorem; cf. \cite{CE}. 

\begin{proposition}\cite[Proposition 1.1.2]{Bo2}
\label{lemma 2.3}
For a triangle $[ABC]$ in $(X, d)$, let $[A'B'C']$ denote a comparison triangle for $[ABC]$. Then $\angle ABC \leq \angle A'B'C', \angle BCA \leq \angle B'C'A'$ and $\angle CAB \leq \angle C'A'B'$. 
\end{proposition}

Proposition \ref{lemma 2.3}  implies  some useful  geometric inequalities in $X$: 

\begin{corollary}
\label{lemma 2.2}
Consider a triangle in $X$ with vertices $ABC$ such that the angles at $A, B, C$ are $\alpha, \beta,  \gamma$ and the sides opposite to $A, B, C$ have lengths $a, b, c$, respectively. If $\gamma \geq \pi /2$, then 
$$\cosh a \sin \beta \leq 1.$$

\end{corollary}
\proof 
Let  $[A'B'C']$  be a comparison triangle for $[ABC]$ in $(\H^{2}, d')$. Let $\alpha ', \beta ', \gamma ' $ denote the angles at $A', B', C'$ respectively as in Figure 1. By Proposition \ref{lemma 2.3},  $d'(A', B')=c, d'(A', C')=b, d'(B',C')=a$ and $\beta'\geq \beta, \gamma'\geq \gamma \geq \pi /2$. Take the point $C''\in A'B'$ such that $\angle B'C'C''=\pi /2$. In the right triangle $[B'C'C'']$ in $\H^{2}$,   we have $\cosh a \sin \beta '=\cos (\angle C'C''B')$, see \cite[Theorem 7.11.3]{AB}. So we obtain the inequality: 
$$\cosh a \sin \beta \leq \cosh a \sin \beta '\leq 1. \qedhere$$


\begin{remark}
\label{boundary inequality 1}
If $A\in \partial_{\infty}X$, we use a sequence of triangles in $X$ to approximate the triangle $[ABC]$ and prove that $\cosh a \sin \beta \leq 1$ still holds by continuity. 

\end{remark}

\begin{figure}[H]
\centering
\includegraphics[width=3.5in]{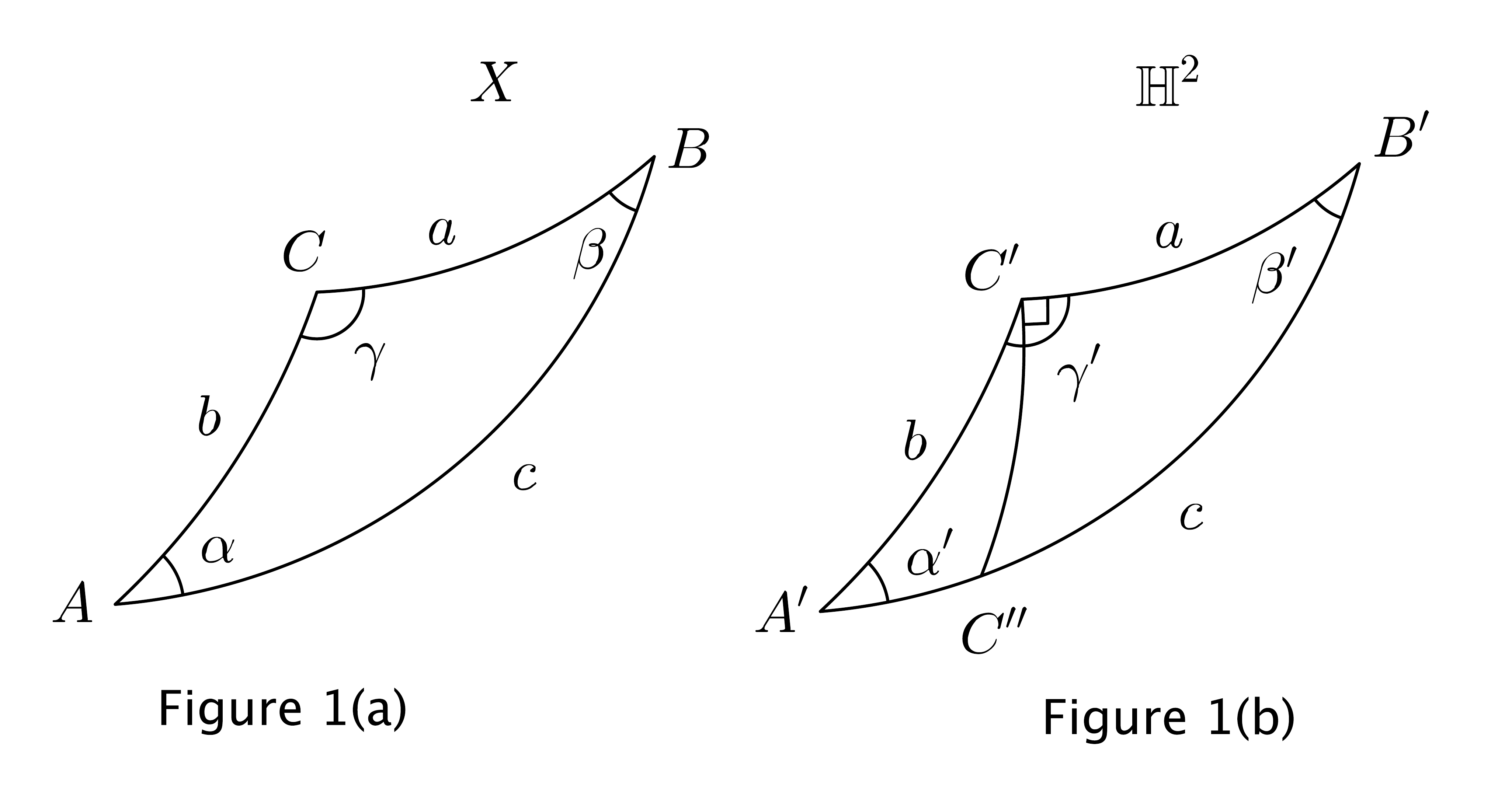} 
\caption{}
\end{figure}

\begin{corollary}
\label{lemma 2.5}
Let $[ABCD]$ denote a quadrilateral in $X$ such that $\angle ABC\geq \pi /2, \angle BCD \geq \pi /2$ and $ \angle CDA\geq \pi /2$  as in Figure 2(a). Then:
\begin{enumerate}
\item 
$\sinh (d(B, C)) \sinh (d(C, D))\leq 1.$

\item 
Suppose that  $\angle BAD\geq \alpha>0$. If $\cosh(d(A, B))\sin \alpha>1$, then $$\cosh(d(C, D))\geq \cosh(d(A, B))\sin \alpha>1.$$ 
\end{enumerate}
\end{corollary}

\proof Let $[A'B'C'D']$ be a comparison quadrilateral for $[ABCD]$ in  $(\H^{2}, d')$ such that $[A'B'C']$ is a comparison triangle for $[ABC]$ and $[A'C'D']$ is a comparison triangle for $[ACD]$. By Proposition \ref{lemma 2.3}, $\angle A'B'C' \geq \pi /2$, $\angle A'D'C'\geq \pi /2$ and 
$$\angle B'C'D'=\angle B'C'A'+\angle A'C'D'\geq \angle BCD \geq \pi/2. $$
Thus, $0<\angle B'A'D'\leq \pi/2$ and $[A'B'C'D']$ is an embedded convex quadrilateral.

We first prove that $\sinh d(B, C) \sinh (d(C, D))\leq 1$. In Figure 2(c), take the point $H \in A'B'$ such that $\angle HC'D'=\pi /2$ and take the point $G\in A'H$ such that $\angle GD'C'=\pi /2$. We claim that $\angle C'HA'\geq \pi /2$. Observe that 
$$\angle C'HB'+\angle HB'C'+ \angle B'C'H \leq \pi$$
$$\angle C'HA'+ \angle C'HB'=\pi .$$
Thus $\angle C'HA'\geq \angle C'B'H\geq \pi /2$. We also have $d'(C',H)\geq d'(C', B')$ since $$\dfrac{ \sinh (d'(C', H))}{\sin (\angle C'B'H)}  = \dfrac{ \sinh (d'(C', B')) }{ \sin (\angle C'HB')} .  $$
Take the point $H'\in GD'$ such that $\angle C'HH'=\pi /2$. In the quadrilateral $[C'HH'D']$, $\cos (\angle HH'D')=\sinh (d'(H, C'))\sinh (d'(C', D'))$, \cite[Theorem 7.17.1]{AB}. Thus, we have

\begin{equation*}
\begin{split}
\sinh (d(C, D))\sinh (d(B, C)) & = \sinh (d'(C', D') \sinh (d'(B', C')) \\
                                              & \leq \sinh (d'(C', D'))\sinh (d'(C', H)) \\
                                              & \leq 1.
\end{split}
\end{equation*}

Next, we prove that if $\cosh(d(A, B))\sin \alpha>1$, then $\cosh(d(C, D))\geq \cosh(d(A, B))\sin \alpha$. In Figure 2(b), take the $C''\in C'D'$ such that $\angle A'B'C''= \pi /2$. Observe that  $C''$ cannot be on $ A'D'$. Otherwise in the right triangle $[A'B'C'']$, we have 
$$\cosh(d(A, B)) \sin \alpha \leq \cosh(d'(A', B')) \sin (\angle B'A'D')\leq 1,$$
which is a contradiction. Let $EF$ denote the geodesic segment which is orthogonal to $B'E$ and $A'F$. In the quadrilateral $[A'B'EF]$, $\cosh (d'(E, F))=\cosh (d'(A', B')) \sin (\angle B'A'F)$ by hyperbolic trigonometry \cite[Theorem 7.17.1]{AB}. Thus, 
$$\cosh(d(C, D))\geq \cosh (d'(C'', D'))\geq \cosh(d'(E, F))\geq \cosh(d(A, B))\sin \alpha. \qedhere$$

\begin{remark}
\label{boundary inequality 2}
If $A\in \partial_{\infty} X$ and $\angle BAD=0$, we  use quadrilaterals in $X$ to approximate the quadrilateral $[ABCD]$ and prove that $\sinh(d(B, C))\sinh(d(C, D))\leq 1$ by continuity. 
\end{remark}

\begin{figure}[H]
\centering
\includegraphics[width=1.2in]{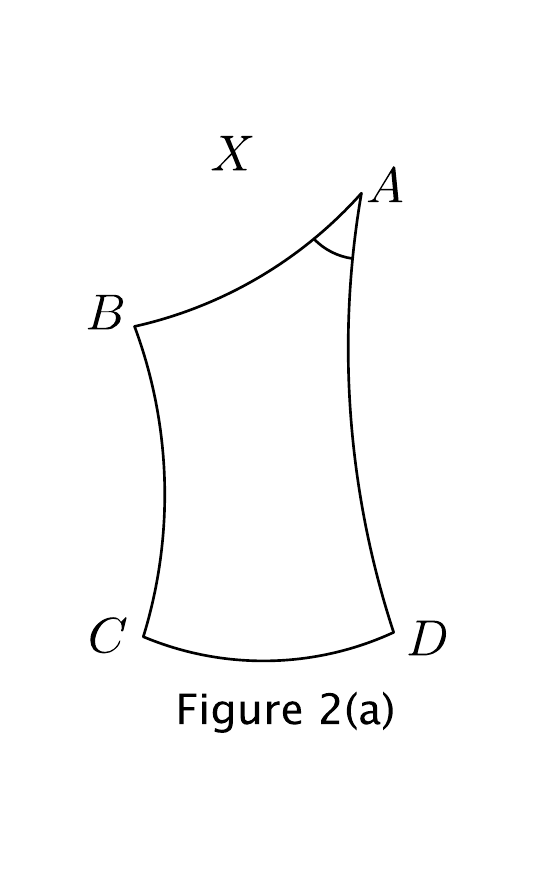}
\hspace{0.3in}
\includegraphics[width=1.2in]{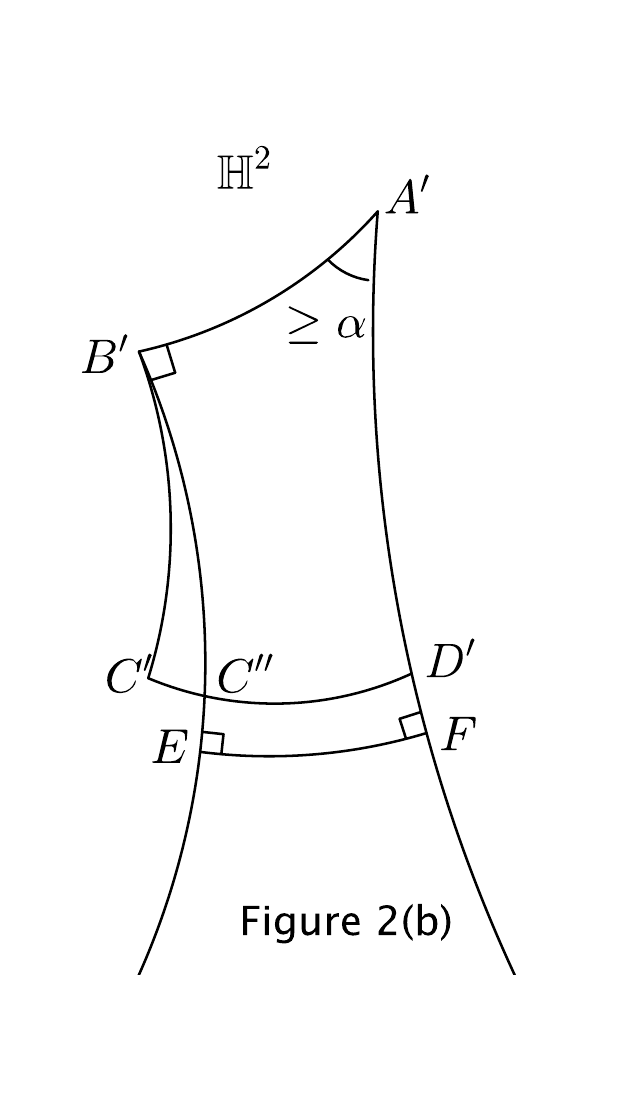} 
\hspace{0.3in}
\includegraphics[width=1.2in]{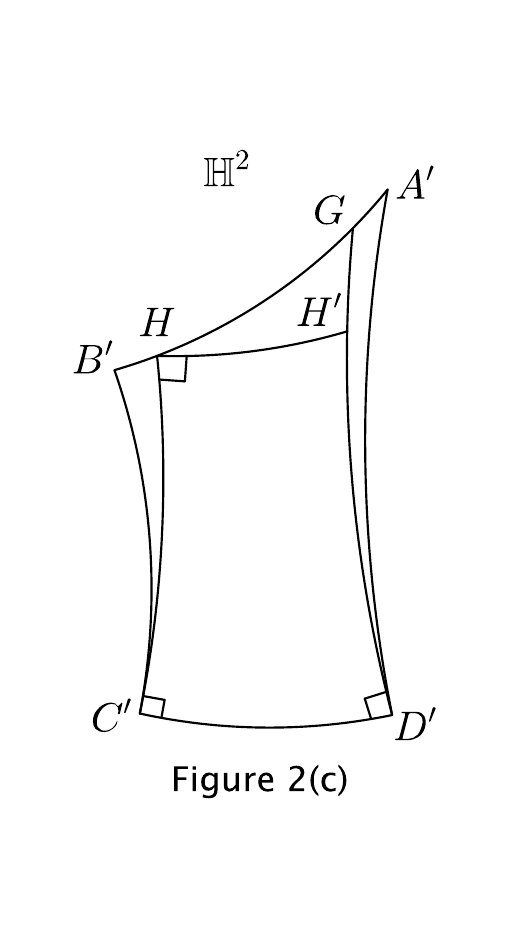}

\caption{}
\end{figure}

Another comparison theorem, {\em the CAT(-1) inequality}, can be used to derive the following proposition (see \cite{Bo2}):

\begin{proposition}\cite[Lemma 2.2.1]{Bo2}
\label{lemma 2.7}
For any $m+1$ points  $x_{0}, x_{1}, \cdots, x_{m}\in \bar{X}$  we have  
$$x_{0}x_{m} \subseteq \bar{N}_{\lambda}(x_{0}x_{1}\cup x_{1}x_{2}\cup \cdots \cup x_{m-1} x_{m})$$
where $\lambda=\lambda_{0}\lceil \log_{2} m \rceil, \lambda_{0}=\cosh^{-1}(\sqrt{2})$. \end{proposition}

\medskip 
Given a point $\xi \in \geo X$, for any point $y\in X$, we use a map $\rho_{y}: \mathbb{R}^{+}\rightarrow X$ to parametrize the geodesic $y\xi$ by its arc-length. The following lemma is  deduced from the $CAT(-1)$ inequality, see \cite{Bo2}:

\begin{lemma}\cite[Proposition 1.1.11]{Bo2}
\label{prop 1.1}
\begin{enumerate}
\item 
Given any $y, z \in X$, the function $d(\rho_{y}(t), \rho_{z}(t))$ is monotonically decreasing in $t$.

\item
For each $r$, there exists a constant $R=R(r)$, such that if $y, z\in X$ lie in the same horosphere about  $\xi$ and  $d(y, z)\leq r$, then $d(\rho_{y}(t), \rho_{z}(t))\leq Re^{-t}$ for all $t$. 
\end{enumerate}
\end{lemma}

\subsection{Convex and quasiconvex subsets}

\begin{definition}\label{def:qc}
A  subset $A\subseteq X$ is  \emph{convex} if $xy\subseteq A$ for all $x, y\in A$. A closed subset $A\subseteq X$ is $\lambda$-\emph{quasiconvex} if $xy\subseteq \bar{N}_{\lambda}(A)$ for all $x, y\in A$. Convex closed subsets are $0$-quasiconvex. 

\end{definition}

\begin{remark}
If $A$ is a $\lambda$-quasiconvex set, then $\textup{QHull}(A) \subseteq \bar{N}_{\lambda}(A)$. 
\end{remark}

\begin{proposition}\cite[Proposition 2.5.4]{Bo2}
\label{convex}
There is a function ${\mathfrak r}_{\kappa}: \R_+ \rightarrow \R_+$ (depending also on $\kappa$) such that for every $\lambda$-quasiconvex subset $A\subseteq X$, we have
$$
\textup{Hull}(A)\subseteq \bar{N}_{{\mathfrak r}_{\kappa}(\lambda)}(A).
$$
\end{proposition}

\begin{remark}
\label{uniformly constant for convex}
Note that, by the definition of the hyperbolicity constant $\delta$ of $X$, the quasiconvex hull 
$\textup{QHull}(A)$ is $2\delta$-quasiconvex for every closed subset $A\subseteq \bar{X}$. Thus, $\Hull(A)\subseteq \bar{N}_{r}(\textup{QHull}(A))$ for some absolute constant 
$r\in [0, \infty)$. 
\end{remark}

\begin{remark}
For any closed subset $A\subseteq \partial_{\infty}X$ with more than one point, $\partial_{\infty} \textup{Hull}(A)=A$. 

\end{remark}

\begin{lemma}
\label{hyper}
Assume that $\xi, \eta$ are  distinct points in $\geo X$ and  $( x_{i})$, $(y_i)$ are sequences in $X$  converging to $\xi$ and to $\eta$ respectively. 
Then for every  point $p\in \xi \eta\subseteq X$, $p\in \bar{N}_{2\delta}(x_{i}y_{i})$ for all sufficiently large $i$. 
\end{lemma}

\begin{proof}
Since $( x_{i} )$ converges to $\xi$ and $( y_{i} )$ converges to $\eta$, we have  
$d(p, x_{i} \xi)\rightarrow \infty$ and $d(p, y_{i}\eta)\rightarrow \infty$ as $i \rightarrow \infty$. By the $\delta$-hyperbolicity of $X$, 
$$p\in \bar{N}_{2\delta}(x_{i}y_{i} \cup x_{i} \xi \cup y_{i} \eta).$$
Since  $d(p, x_{i} \xi)\rightarrow \infty$ and  $d(p, y_{i} \eta)\rightarrow \infty$, we obtain  
$$p\in \bar{N}_{2\delta}(x_{i} y_{i})$$ 
for sufficiently large $i$. 
\end{proof}

\begin{remark}
This lemma holds for any $\delta$-hyperbolic geodesic metric space. 
\end{remark}

\subsection{Volume inequalities} 

Let $V(r, n)$ denote the volume of the  $r$-ball in $\H^{n}$. Then, for a positive constant $c_n$ depending only on $n$, 
$$
V(r,n)= c_n \int_0^r \sinh^{n-1}(t)dt \le \frac{c_n}{2^{n-1}(n-1)} e^{(n-1)r} = C_n   e^{(n-1)r},  
$$
see e.g. \cite[Sect. 1.5]{Nicholls}.

Volumes of metric balls $B(x, r)\subset X$ satisfy the inequalities 
\begin{equation}\label{eq:volume-inequalities} 
V(r, n)\le Vol B(x,r) \le V(\kappa r, n)/\kappa^{n},   
\end{equation}
see e.g. Proposition 1.1.12 and Proposition 1.2.4 in \cite{Bo2}, or \cite[Sect. 11.10]{BiCr}. As a corollary of these volume inequalities we obtain the following {\em packing inequality}:

\begin{lemma}\label{lem:packing}
Suppose that $Z\subset X$ is a subset  such that the minimal distance between distinct points of $Z$ is at least $2r$. Then for every $x\in X$, $R\ge 0$, we have 
$$
\card (B(x,R)\cap Z)\le \frac{V(\kappa(R+r), n)}{\kappa^{n} V(r, n)}\leq   \frac{C_n }{\kappa^{n} V(r, n)}  e^{\kappa(n-1)(R+r)}.   
$$
In particular, if 
$$\card (Z)>  \frac{C_n}{V(r, n)}  e^{\kappa(n-1)(R+r)}$$ then for any $z\in Z$ there exists $z'\in Z$ such that $d(z,z')> R$. 
\end{lemma}

\section{Escaping sequences of closed geodesics in negatively curved manifolds}
\label{sec:escaping}

In this section, $X$ is a Hadamard manifold of negative curvature $\le -1$ with the hyperbolicity constant $\delta$, $\Gamma< \Isom(X)$ is a discrete isometry subgroup  and $M=X/\Gamma$ is the quotient orbifold. A sequence of subsets $A_i\subset M$ 
 is said to {\em escape every compact subset of $M$} if 
for every compact $K\subset M$, the subset 
$$
\{i\in \N: A_i\cap K\ne \emptyset\}
$$
is finite. Equivalently, for every $x\in M$, $d(x, A_i)\to \infty$ as $i\to \infty$.

\begin{lemma}
\label{escape compact set}
Suppose that $(a_i)$ is a sequence of closed geodesics in $M=X/ \Gamma$ which escapes every compact subset of $M$ and $x\in M$. Then, after 
passing to a subsequence in $(a_i)$, there exist geodesic arcs $b_i$ connecting $a_i, a_{i+1}$ and orthogonal to these geodesics, such that the sequence $(b_i)$ also 
escapes every compact subset of $M$. 
\end{lemma}

\begin{proof}

Consider a sequence of compact subsets $K_n:= \bar{B}(x, 7\delta n)$ exhausting $M$. Without loss of generality, we may assume that $a_i\cap K_n=\emptyset$ for all $i\ge n$. 

We first prove the following claim: 

\begin{claim*}
For each compact subset $K\subset M$ and for each infinite subsequence $(a_i)_{i\in I}, I\subset \N$, there exists a further infinite subsequence, 
$(a_i)_{i\in J}, J\subset I$, such that for each pair of distinct elements $i, j\in J$, there exists a geodesic arc $b_{ij}$ connecting 
$a_i$ to $a_j$ and orthogonal to both, which is disjoint from $K$.  
\end{claim*}
\proof  Given two closed geodesics $a, a'$ in $M$, we consider the set $\pi_{1}(M, a, a')$ of {\em relative homotopy classes} of paths in $M$ 
connecting $a$ and $a'$, where the relative homotopy is defined through paths connecting $a$ to $a'$. 

In each class $[b']\in \pi_{1}(M, a, a')$, there exists a continuous path $b$ which is the length minimizer in the class. 
By minimality of its length, $b$ is a geodesic arc orthogonal to $a$ and $a'$ at its end-points. 

For each compact subset $K\subset M$, there exists $m\in \N$ such that for all $i\in I_m:=I\cap [m,\infty)$, $a_i\cap K'=\emptyset$ where $K'=\bar{N}_{7\delta}(K)$.
For $i\in I_m$ let $c_{i}$ denote a shortest arc between $a_{i}$ and $K'$; this geodesic arc terminates a point $x_i\in K'$. 
By compactness of $K'$, the sequence $(x_i)_{i\in I_m}$ contains a convergent subsequence,  $(x_i)_{i\in J}, J\subset I_m$ and, without loss of generality, we may assume 
that for all $i, j\in J$, $d(x_i, x_j)\le \delta$. Let $x_i x_j$ denote a (not necessarily unique) geodesic in $M$ of  length $\le \delta$ connecting $x_i$ to $x_j$.  
 For each pair of indices $i, j\in J$, consider the concatenation 
 $$
 b'_{ij}=c_{i}\ast x_{i}x_{j} \ast c_{j}^{-1},$$
 which defines a class $[b'_{ij}]\in \pi_{1}(M, a_{i}, a_{j})$.  
Let $b_{ij}\in [b'_{ij}]$ be a length-minimizing geodesic arc in this relative homotopy class. 
Then $b_{ij}$ is orthogonal to $a_{i}$ and $a_{j}$. By the $\delta$-hyperbolicity of $X$, 
$$
b_{ij}\subseteq \bar{N}_{7\delta}(a_{i}\cup c_{i} \cup c_{j}\cup a_{j}).$$
Hence,  $b_{ij}\cap K=\emptyset$ for any pair of distinct indices $i, j\in J$. This proves the claim. \qed

\medskip 
We now prove the lemma. Assume inductively (by induction on $N$) that we have constructed an infinite subset $S_N\subset \N$ such that:

For the $N$-th element $i_N\in S_N$, for each $j > i_N, j\in S_{N}$,  there exists a geodesic arc $b_{j}$ in $M$ 
connecting $a_{i_N}$ to $a_{j}$ and orthogonal to both, which is disjoint from $K_{N-1}$. 

Using the claim, we find an infinite subset $S_{N+1}\subset S_N$ which contains the first $N$ elements of $S_N$, such that for all $s, t> i_{N}, s, t\in S_{N+1}$, there exists a geodesic $b_{s,t}$ in $M$ connecting $a_{s}$ to $a_{t}$, orthogonal to both and disjoint from $K_N$. 

The intersection 
$$
S:= \bigcap_{N\in \N} S_N
$$
equals $\{i_N: N\in \N\}$ and, hence, is infinite. 
We, therefore, obtain a subsequence $(a_i)_{i\in S}$ such that for all $i, j\in S, i < j$, 
there exists a geodesic $b_{ij}$ in $M$ connecting $a_i$ to $a_j$ and orthogonal to both, which is disjoint from $K_{i-1}$.  
\end{proof}

\begin{remark}
It is  important to pass a subsequence of $(a_{i})$, otherwise, the lemma is false. A counter-example is given by a geometrically infinite manifold 
with two distinct ends $E_{1}$ and $E_{2}$ where we have a sequence of closed geodesics $a_{i}$ (escaping every compact subset of $M$) contained in 
$E_{1}$ for odd $i$ and in $E_{2}$ for even $i$. Then $b_{i}$ will always intersect a compact subset separating the two ends no matter what $b_{i}$ we take. 
\end{remark}

\section{Elementary groups of isometries}
\label{sec:elementary}

Every isometry $g$ of $X$ extends  to a homeomorphism (still denoted by $g$) of $\bar{X}$. We let  $\textup{Fix}(g)$ denote the fixed point set of $g: \bar X\to \bar X$. For a subgroup 
$\Gamma< \Isom(X)$, we use the notation
$$
\textup{Fix}(\Gamma):= \bigcap_{g\in \Gamma} \textup{Fix}(g),
$$
to denote the fixed point set of $\Gamma$ in $\bar X$. Typically, this set is empty.

\medskip 
Isometries of $X$ are classified as follows: 


\begin{enumerate}

\item $g$ is {\em parabolic} if  $\textup{Fix}(g)$  is a singleton $\{p\}\subset \geo X$. In this case, $g$ 
 preserves (setwise) every horosphere centered at 
 $p$. 

\item $g$ is loxodromic if  $\textup{Fix}(g)$ consists of two distinct points $p, q \in \geo X$. The loxodromic isometry $g$ preserves the geodesic $pq\subset X$ and acts on it as a nontrivial  translation. The geodesic $pq$ is called the {\em axis} $A_g$ of $g$.  

\item $g$ is elliptic if 
it fixes a point in $X$.  
The fixed point set of an elliptic isometry is a totally-geodesic subspace of $X$ invariant under $g$. In particular, the identity map is an elliptic isometry of $X$. 


\end{enumerate}

If $g\in \Isom(X)$ is such that  $\textup{Fix}(g)$ contains three distinct points $\xi, \eta, \zeta\in \geo X$, then $g$ also fixes pointwise the convex hull $\Hull(\{\xi, \eta, \zeta\})$ and, hence, $g$ is an elliptic isometry of $X$. 

For each isometry $g\in \Isom(X)$ we define its translation length $l(g)$ as follows:
$$l(g)=\inf\limits_{x\in X} d(x, g(x)), $$
and we define the \emph{rotation} of $g$ at $x\in X$ as:
$$r_{g}(x)=\max \limits_{v\in T_{x}X} \angle (v, P_{g(x), x}\circ g_{\ast_{x}} v). $$
Here $g_{\ast_{x}}: T_{x}X\rightarrow T_{g(x)}X$ is the differential and $P_{g(x), x}: T_{g(x)}X\rightarrow T_{x}X$ is the parallel transport along the unique geodesic from $g(x)$ to $x$. Following \cite{BaGS}, given $a\geq 8$ we define the \emph{norm of} $g$ at $x$ as $n_{g}(x)=\max (r_{g}(x), a\cdot d_{g}(x))$ where $d_{g}(x)=d(x, g(x))$.


A discrete subgroup $G$ of isometries of $X$ is called \emph{elementary} if either $\textup{Fix}(G)\neq \emptyset$ or
 if $G$ preserves set-wise some bi-infinite geodesic in $X$. (In the latter case, $G$ contains an index 2 subgroup $G'$ such that  $\textup{Fix}(G')\neq \emptyset$.) Based on the fixed point set, elementary groups are divided into the following three classes \cite{Bo2}: 

\begin{enumerate}

\item $F(G)$ is a nonempty subspace of $\bar{X}$. 

\item $F(G)$ consists of a single point of $\geo X$. 

\item $G$ has no fixed point in $X$, and $G$ preserves setwise a unique bi-infinite geodesic in $X$. 

\end{enumerate}

\begin{remark}
\label{elementary groups}
If $G<\Isom(X)$ is discrete and in the first class, then $G$ is finite by discreteness and consists of elliptic isometries. If $G$ is discrete and in the second class, it is called parabolic, and it contains a parabolic isometry \cite[Proposition 4.2]{Bo2}. Discrete groups $G$ in the third class will be called  elementary loxodromic groups. 
\end{remark}

\begin{lemma}
\label{finite elementary group}
If $G<\Isom(X)$ is a discrete elementary subgroup consisting entirely of elliptic elements, then $G$ is finite. 
\end{lemma}
\begin{proof}
By Remark \ref{elementary groups}, $G$ is either finite or loxodromic. Suppose that $G$ is loxodromic and preserves a geodesic $l\subset X$ setwise. Let $\rho: G\to \Isom(l)$ denote 
the restriction homomorphism. Since $G$ is loxodromic, the subgroup $\rho(G)$ has no fixed point in $l$. Hence, there exist two elements $g, h\in G$ such that $\rho(g), \rho(h)$ are 
distinct involutions. Their product $\rho(g) \rho(h)$ is a nontrivial translation of $l$. Hence, $gh$ is a loxodromic isometry of $X$, contradicting our assumption. 
Hence, $G$ is finite. 
\end{proof}

\begin{corollary}
Every discrete elementary loxodromic group  contains a loxodromic isometry. 
\end{corollary}

Consider a  subgroup $\Gamma$ of isometries of $X$. Given any subset $Q\subseteq \bar X$, let 
$$
\stab_{\Gamma}(Q)=\lbrace \gamma \in \Gamma \mid \gamma(Q)=Q \rbrace$$ 
denote the setwise stabilizer of $Q$ in $\Gamma$. 

\begin{definition}
A point $p\in \geo X$ is called a \emph{parabolic fixed point} of a subgroup $\Gamma< \Isom(X)$ if $\stab_{\Gamma}(p)$ is parabolic. 
\end{definition}

\begin{remark}
If $p\in \geo X$ is a parabolic fixed point of a discrete subgroup $\Gamma<  \Isom(X)$, then $\stab_{\Gamma}(p)$ is a maximal parabolic subgroup of $\Gamma$, see 
\cite[Proposition 3.2.1]{Bo2}. Thus, we have a bijective correspondence between the 
$\Gamma$-orbits of parabolic fixed points of $\Gamma$ and the $\Gamma$-conjugacy classes of maximal parabolic subgroups of $\Gamma$. 
\end{remark}

Consider an elementary loxodromic subgroup $G< \Gamma$ with the axis $\beta$. Then $\stab_{\Gamma}(\beta)$ is a maximal loxodromic subgroup of $\Gamma$, see \cite[Proposition 3.2.1]{Bo2}.




\section{The thick-thin decomposition}  
\label{sec:thick-thin}

For an isometry $g\in \Isom(X)$, define the \emph{Margulis region} $Mar(g, \varepsilon)$ of $g$  as:
$$Mar(g, \varepsilon)=\lbrace x\in X \mid d(x, g(x))\leq \varepsilon \rbrace .$$ 
By the convexity of the distance function, $Mar(g, \varepsilon)$ is convex.  

Given $x\in X$ and a discrete subgroup $\Gamma< \Isom(X)$, let $\mathcal{F}_{\varepsilon}(x)=\lbrace \gamma\in \Gamma \mid d(x, \gamma x)\leq \varepsilon \rbrace$ denote the set of isometries in $\Gamma$ which move $x$ a distance at most $\varepsilon$. Let $\Gamma_{\varepsilon}(x)$ denote the subgroup generated by $\mathcal{F}_{\varepsilon}(x)$. We use $\varepsilon(n, \kappa)$ to denote the Margulis constant of $X$. Then, by the Margulis Lemma, $\Gamma_{\varepsilon}(x)$ is virtually nilpotent whenever 
$0<\varepsilon\leq \varepsilon(n, \kappa)$. More precisely,


\begin{proposition}\cite[Theorem 9.5]{BaGS}
\label{nilpotent subgroup with finite index}
Given $0<\varepsilon \leq \varepsilon (n, \kappa)$ and $x\in X$, the group $N$ generated by the set $\lbrace \gamma \in \Gamma_{\varepsilon}(x) \mid n_{\gamma}(x) \leq 0.49 \rbrace$ is a nilpotent subgroup of $\Gamma_{\varepsilon}(x)$ of a uniformly bounded index (where the bound depends only on $\kappa$ and $n$). 
Moreover, each coset $\gamma N\subset \Gamma_{\varepsilon}(x)$ can be represented by an element $\gamma$ of word  length $\leq m(n, \kappa)$ in the 
generating set $\mathcal{F}_{\varepsilon}(x)$ of  $\Gamma_{\varepsilon}(x)$. Here $m(n, \kappa)$ is a constant depending only on $\kappa$ and $n$.   

\end{proposition}

\begin{remark}
$\Gamma_{\varepsilon}(x)$ is always finitely generated.
\end{remark}

We will use the following important property of nilpotent groups in Section \ref{sec:loxodromic}: 

\begin{theorem}\cite{KD, Ku}
\label{torsion group}
Let $G$ be a nilpotent group. The set of all finite order elements of $G$ forms a characteristic subgroup of $G$. This subgroup is called the {\em torsion subgroup} of $G$ and 
denoted by Tor$(G)$. 
\end{theorem}

Given $0<\varepsilon\leq \varepsilon(n, \kappa)$ and a discrete subgroup $\Gamma<\Isom(X)$, define the set 
$$
T_{\varepsilon}(\Gamma)=\lbrace p\in X \mid \Gamma_{\varepsilon}(p) \textup { is infinite} \rbrace.$$
 Below we establish some properties of $T_{\varepsilon}(\Gamma)$ where $\Gamma< \Isom(X)$ are discrete  subgroups. 

\medskip 
Applying Lemma \ref{lem:packing}  to the subset $Z= G\cdot x\subset X$ we obtain:   

\begin{lemma}\label{lem:large displacement} 
Suppose that $G=\langle g\rangle$ is a (discrete) infinite cyclic subgroup and $x\notin \mathrm{int}(T_{\varepsilon}(G))$, i.e. $d(x, g^i(x))\ge \varepsilon$ for all $i\ne 0$. Then for every $D$ there exists $i$, 
$$
0< i\le  N(\varepsilon, n, \kappa, D) :=  1+ \frac{C_n  e^{\kappa(n-1)\varepsilon/2}}{\kappa^n V(\varepsilon/2, n)}  e^{\kappa(n-1)D}
$$
such that $d(x, g^i x)\ge D$. 
\end{lemma}

\begin{lemma}
\label{smaller cusp}
Suppose that $G<\Isom(X)$ is a discrete parabolic subgroup and $\varepsilon>0$. For any $z\in T_{\varepsilon /3}(G)$, we have $B(z, \varepsilon /3)\subseteq T_{\varepsilon}(G)$. 
\end{lemma}

\begin{proof}
The set  $\calF_{\varepsilon/3}(z)=\lbrace \gamma \in G| d(z, \gamma(z))\leq \varepsilon/3 \rbrace$ generates an infinite  subgroup of $G$ since $z\in T_{\varepsilon /3}(G)$. For any element $\gamma\in \calF_{\varepsilon/3}(z)$ and $z'\in B(z, \varepsilon /3)$, we have 
$$d(z', \gamma(z'))\leq d(z, z')+d(z, \gamma(z))+d(\gamma(z), \gamma (z'))\leq \varepsilon/3 + \varepsilon / 3 +\varepsilon /3= \varepsilon.$$
Therefore, $\calF_{\varepsilon}(z')=\lbrace \gamma \in G| d(z', \gamma (z'))\leq \varepsilon \rbrace$ also generates an infinite subgroup. Thus  $z'\in T_{\varepsilon}(G_{i})$ and $B(z, \varepsilon/3)\subseteq T_{\varepsilon}(G)$. 

\end{proof}

\begin{proposition}\cite[Proposition 3.5.2]{Bo2}
\label{starlike}
Suppose $G< \Isom(X)$ is a discrete parabolic subgroup with the fixed point $p\in \geo X$, and  $\varepsilon >0$. Then $T_{\varepsilon}(G) \cup \lbrace p \rbrace$  
is starlike about $p$, i.e. for each $x\in \bar{X}\setminus \lbrace p \rbrace$, 
 the intersection $xp\cap T_{\varepsilon}(G)$ is a ray asymptotic to $p$. 
\end{proposition}

\begin{corollary}
\label{quasiconvex}
Suppose that $G< \Isom(X)$ is a discrete parabolic subgroup with the fixed point $p\in \geo X$. For every $\varepsilon>0$, $T_{\varepsilon}(G)$ is a $\delta$-quasiconvex subset of $X$. 

\end{corollary}

\begin{proof}
By Proposition \ref{starlike}, $T_{\varepsilon}(G) \cup \lbrace p \rbrace$ is starlike about $p$.
Every starlike set is 
$\delta$-quasiconvex, \cite[Corollary 1.1.6]{Bo2}. Thus $T_{\varepsilon}(G)$ is $\delta$-quasiconvex for every  discrete parabolic subgroup $G< \Isom(X)$. 
\end{proof}

\begin{remark}
\label{constant for quasiconvex}
According to Proposition \ref{convex}, there exists  $r={\mathfrak r}_{\kappa}(\delta)\in [0, \infty)$ such that $\Hull(T_{\varepsilon}(G))\subseteq \bar{N}_{r}(T_{\varepsilon}(G))$ for any $\varepsilon>0$.
\end{remark}

\begin{lemma}\label{lem:boundaryofcusp}
If $G<\Isom(X)$ is a discrete parabolic subgroup with the fixed point $p\in \geo X$, then $\geo T_{\varepsilon}(G)=\lbrace p \rbrace$.
\end{lemma}

\begin{proof}
By  Lemma \ref{prop 1.1}(2), for any $p'\in \geo X\setminus \lbrace p \rbrace$, both $p'p\cap T_{\varepsilon}(G)$ and $X \cap( p'p\setminus T_{\varepsilon}(G))$ are nonempty \cite[Proposition 3.5.2]{Bo2}. If $p'\in \geo T_{\varepsilon}(G)$, there exists a sequence of points $( x_{i} ) \subseteq T_{\varepsilon}(G)$ which converges to $p'$. By Proposition \ref{starlike}, $ x_{i}p \subseteq T_{\varepsilon}(G)$. Since $T_{\varepsilon}(G)$ is closed in $X$, then $p'p\subseteq T_{\varepsilon}(G)$, which is a contradiction. 

\end{proof}

\begin{proposition}
\label{enter cusp}
Suppose that $G<\Isom(X)$ is a discrete parabolic subgroup with the fixed point $p\in \geo X$.  Given $r>0$ and $x\in X$ with $d(x, \textup{Hull}(T_{\varepsilon}(G)))=r$, if $( x_{i} )$ is a sequence of points on the boundary of $\bar{N}_{r}(\textup{Hull}(T_{\varepsilon}(G)))$ and  $d(x, x_{i})\rightarrow \infty$, then there exists $z_{i}\in xx_{i}$ such that the sequence $( z_{i} )$ converges to $p$ and for every $\varepsilon>0$, 
$z_{i}\in \bar{N}_{\delta}(T_{\varepsilon}(G))$ for all sufficiently large $i$. 

\end{proposition}

\begin{proof} 

By the $\delta$-hyperbolicity of $X$, there exists a point $z_{i}\in xx_{i}$ such that $d(z_{i}, px)\le \delta$ and $d(z_{i}, px_{i})\le \delta$. Let $w_{i}\in px_{i}$ and $v_{i}\in px$ be the points closest to $z_i$, see  
Figure \ref{converge}. Then $d(z_{i}, w_{i})\leq \delta$, $d(z_{i}, v_{i})\leq \delta$ and, hence, $d(w_{i}, v_{i})\leq 2\delta$. 

According to Lemma \ref{lem:boundaryofcusp}, the sequence $(x_i)$ converges to the point $p$. Hence, any sequence of points on $x_ip$ converges to $p$ as well; in particular, $(w_i)$ converges to $p$. 
As $d(w_i, z_i)\le \delta$, we also obtain 
$$
\lim_{i\to\infty} z_i= p. 
$$

Since $d(z_i, v_i)\le \delta$, it suffices to show that $v_i\in T_{\varepsilon}(G)$ for all sufficiently large $i$. This follows from the fact that $d(x, v_i)\to\infty$ and  that $xp\cap T_{\varepsilon}(G)$ is a geodesic ray asymptotic to $p$.

\end{proof}

\begin {figure}[H]
\centering
\includegraphics[width=4.0in]{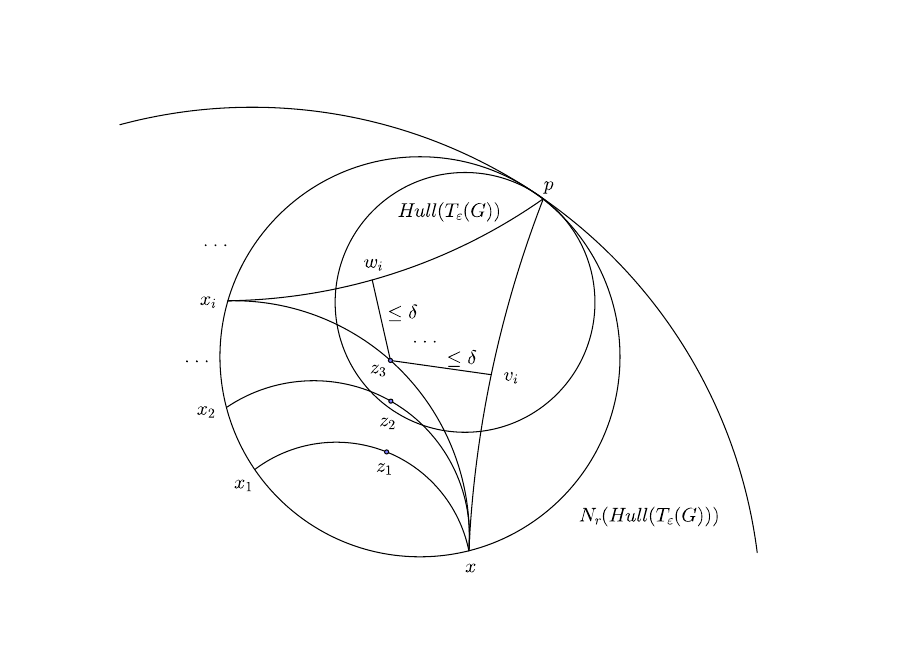}

\caption{ \label{converge}}
\end{figure}

\begin{proposition}
\label{free homotopic2}
Let $\langle g \rangle < \Isom(X)$ be the cyclic group generated by a loxodromic isometry $g$. Let $\gamma$ denote the simple closed geodesic $A_{g}/ \langle g \rangle$ in 
$M =X/ \langle g \rangle$. If $w \subseteq M$ is a piecewise-geodesic loop freely homotopic to $\gamma$ which consists of $r$ geodesic segments, then 
$$d(w, \gamma\cup (Mar(g, \epsilon)/ \langle g\rangle))\leq \cosh^{-1}(\sqrt{2})\lceil \log_{2} r \rceil+\sinh^{-1}(2/\epsilon).$$

\end{proposition}

\begin{proof}
Let $x\in w$ be one of the vertices. Connect this point to itself by a geodesic segment $\alpha$ in $M$ which is homotopic to $w$ (rel $\lbrace x \rbrace$). The loop $w \ast \alpha^{-1}$ lifts to a polygonal loop $\beta\subseteq X$ with consecutive vertices $x_{0}, x_{1}, \cdots, x_{r}$ such that the geodesic segment $\tilde{\alpha}:=x_{0}x_{r}$  covers $\alpha$. Let $\tilde{w}$ denote the union of edges of $\beta$ distinct from $\tilde{\alpha}$. By Proposition \ref{lemma 2.7}, $\tilde{\alpha} $ is contained in the $\lambda$-neighborhood of the piecewise geodesic path $\tilde{w}$ where $\lambda=\cosh^{-1}(\sqrt{2}) \lceil \log_{2} r \rceil$. It follows that $\alpha \subseteq \bar{N}_{\lambda}(w)$. 

Suppose that $Mar(g, \epsilon)\neq \emptyset$. It is closed and convex. Let $h=d(\tilde{\alpha}, Mar(g, \varepsilon))$. Choose  points $A\in \tilde{\alpha}, B\in Mar(g, \epsilon)$ such that $d(A, B)=h$ realizes the minimal distance between $\tilde{\alpha}$ and $Mar(g, \epsilon)$.  Let $F=\textup{Proj}_{Mar(g, \epsilon)}(x_{r})$.  Then we obtain a quadrilateral $[ABFx_{r}]$ with $\angle ABF=\angle BFx_{r} =\angle BAx_{r}\geq \pi /2$. By Corollary \ref{lemma 2.5}, 
$$d(B, F)\leq \sinh (d(B, F))\leq 1/ \sinh (h). $$


Take the point $D\in Mar(g, \epsilon)$ which is closest to $x_{0}$. By a similar argument, we have $d(B, D)\leq 1/ \sinh (h)$. Thus, $d(F, D)\leq 2/ \sinh (h)$. The projection 
$Proj_{A_{g}}$ is $\langle g \rangle$-equivariant, thus $F, D$ are identified by the isometry $g$. Hence
$$\epsilon = d(D, g(D))=d(D, F)\leq 2/ \sinh (h) $$
and  $h\leq \sinh^{-1}(2/\epsilon)$. 

If $Mar(g, \epsilon)=\emptyset$, then the translation length $l(g)\geq \epsilon$. Let $h=d(\tilde{\alpha}, A_g)$. Replacing $Mar(g, \epsilon)$ by $A_{g}$, we use a similar argument to obtain that 
$$d(A_g, \tilde{\alpha})\leq \sinh^{-1}(2/\epsilon).$$  Hence, 
$$d(w, \gamma\cup (Mar(g, \epsilon)/ \langle g\rangle ))\leq \cosh^{-1}(\sqrt{2})\lceil \log_{2} r \rceil+\sinh^{-1}(2/\epsilon).$$

\end{proof}

\begin {figure}[H]
\centering
\includegraphics[width=5.00in]{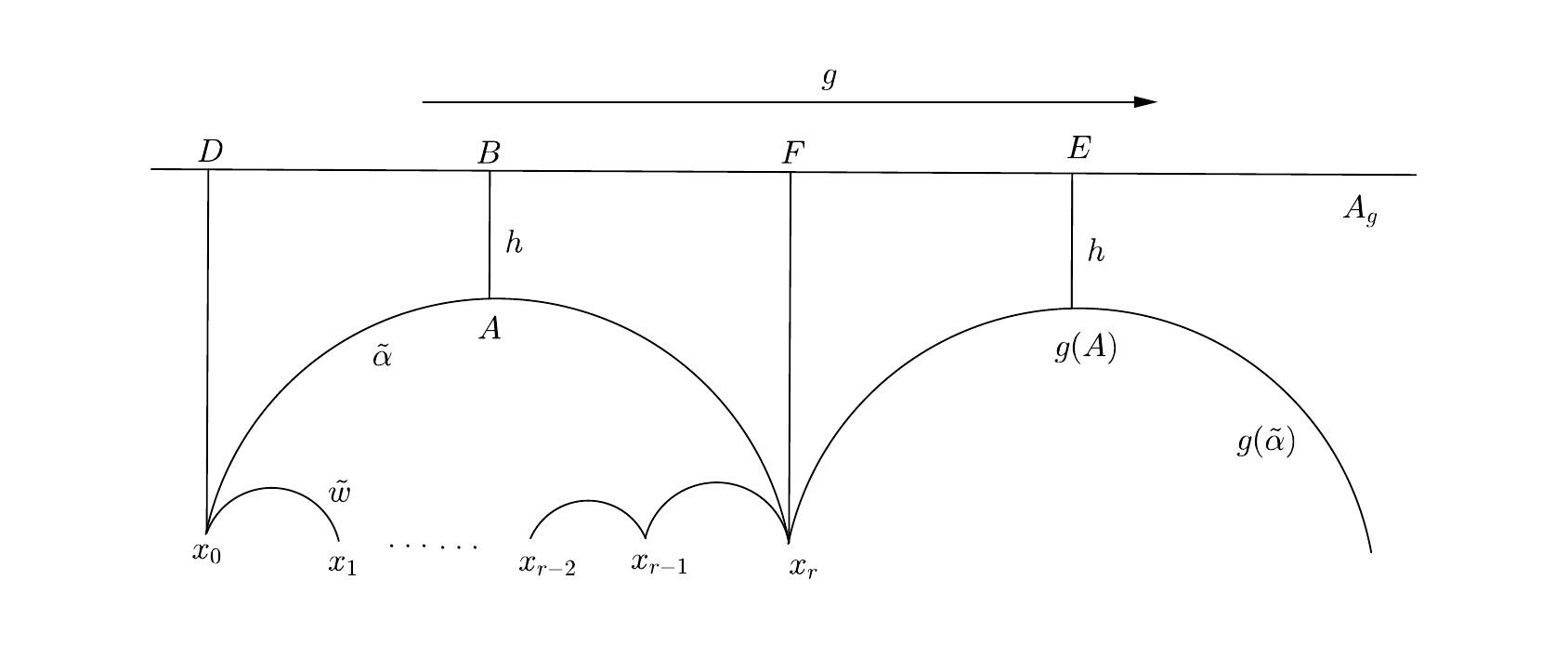}

\caption{  \label{loxnear}}
\end{figure}

\begin{corollary}
\label{free homotopic}
Under the conditions in Proposition \ref{free homotopic2}, if the translation length of $g$ satisfies that $l(g)\geq \epsilon>0$, then $\gamma$ is contained in the $C$-neighborhood of the loop  $w$ where $C=\cosh^{-1}(\sqrt{2})\lceil \log_{2} r \rceil+\sinh^{-1}(2/\epsilon)+2\delta$.

\end{corollary}

\begin{proof}
We use the same notations as in  proof of Proposition \ref{free homotopic2}. Let $E\in A_{g}$ be the nearest point to $g(A)$ as in Figure \ref{loxnear}. Then $\pi(BE)$ in $M=X/ \langle g \rangle$ is the geodesic loop $\gamma$ where $\pi$ is the covering projection. By $\delta$-hyperbolicity of $X$, $BE$ is within the $(h+2\delta)$-neighborhood of the lifts of $\alpha$ as in Figure \ref{loxnear}. Thus $\gamma$ is within the $( \sinh^{-1}(2/\epsilon)+2\delta)$-neighborhood of $\alpha$. Since $\alpha$ is contained in the $(\cosh^{-1}(\sqrt{2}) \lceil \log_{2} r \rceil )$-neighborhood of $w$, the loop $\gamma$ is contained in the $(\cosh^{-1}(\sqrt{2}) \lceil \log_{2} r \rceil+\sinh^{-1}(2/\epsilon)+2\delta)$-neighborhood of $w$. 

\end{proof}

Given $0<\varepsilon\leq  \varepsilon(n, \kappa)$ and a discrete  subgroup $\Gamma$, the set 
$T_{\varepsilon}(\Gamma)$ is a disjoint union of the subsets of the form $T_{\varepsilon}(G)$,  
where $G$ ranges over all maximal infinite elementary subgroups of $\Gamma$, \cite[Proposition 3.5.5]{Bo2}. If $G<\Gamma$ is a maximal parabolic subgroup, $T_{\varepsilon}(G)$ is precisely invariant and Stab$_{\Gamma}(T_{\varepsilon}(G))=G$, \cite[Corollary 3.5.6]{Bo2}. In this case, by abuse of notation, we regard $T_{\varepsilon}(G)/G$ as a subset of $M$, and call it a \emph{Margulis cusp}. Similarly, if $G<\Ga$ is a maximal loxodromic subgroup, $T_{\varepsilon}(G) / G$ is called a \emph{Margulis tube}. 

For the quotient orbifold $M=X/\Gamma$, set 
$$\textup{thin}_{\varepsilon}(M)=T_{\varepsilon}(\Gamma) / \Gamma.$$
This closed subset is the {\em thin part}\footnote{more precisely, $\varepsilon$-thin part}  of the quotient orbifold $M$. It is a disjoint union of its connected components, and  each such component has the form $T_{\varepsilon}(G)/ G$, where $G$ ranges over all maximal infinite elementary subgroups of $\Gamma$. 

The closure of the complement $M\setminus \textup{thin}_{\varepsilon}(M)$ is the {\em thick part} of $M$, denoted by $\textup{thick}_{\varepsilon}(M)$. Let $\textup{cusp}_{\varepsilon}(M)$ denote the union of all Margulis cusps of $M$;  it is called the \emph{cuspidal} part of $M$. The closure of the complement $M\setminus \textup{cusp}_{\varepsilon}(M)$ is denoted by $\textup{noncusp}_{\varepsilon}(M)$; it is called the \emph{noncuspidal} part of $M$. Observe that $\textup{cusp}_{\varepsilon}(M)\subseteq \textup{thin}_{\varepsilon}(M)$ and $\textup{thick}_{\varepsilon}(M)\subseteq \textup{noncusp}_{\varepsilon}(M)$. If $M $ is a manifold (i.e., $\Gamma$ is torsion-free), the $\varepsilon$-thin part is also the collection of all points $x\in M$ where the injectivity radius of $M$ at $x$ is no greater than $\varepsilon /2 $.

\section{Quasigeodesics}
\label{sec:quasigeodesics}

In this section, $X$ is a Hadamard manifold of sectional curvature $\leq -1$. We will prove that certain concatenations of geodesics in $X$ are uniform quasigeodesics, therefore, according to the Morse Lemma, are uniformly close to geodesics.

\begin{definition}
A map $q: I \to X$ defined on an interval $I\subset \R$ is called a $(\la,\alpha)$-quasigeodesic (for $\la\ge 1$ and $\alpha\ge 0$)  if 
$$
 \la^{-1} |s-t| - \alpha \le d(q(s), q(t))\le \la |s-t| + \alpha
$$
for all $s, t\in I$. 
\end{definition}

\begin{proposition}[Piecewise-geodesic paths with long edges] 
\label{piecewise geodesic path}
Define the function 
$$L(\theta)=2\cosh^{-1}\left(\dfrac{2}{\sin (\theta/2)}\right)+1.$$
Suppose that $\gamma=\gamma_{1}\ast \cdots \ast \gamma_{n}\subseteq \bar{X}$ is a piecewise geodesic path\footnote{parameterized by its arc-length}  
from $x$ to $y$ where each $\gamma_{i}$ is a geodesic  of length 
$\ge L=L(\theta)$ and the angles between adjacent arcs $\gamma_{i}$ and $\gamma_{i+1}$ 
are $\ge \theta>0$. Then $\gamma$ is a $(2L, 4L+1)$-quasigeodesic. 
\end{proposition}

\begin{proof}
Recall that $\textup{Bis}(x_{i}, x_{i+1})$ denotes the perpendicular  bisector of $\gamma_{i}=  x_{i} x_{i+1}$ where $x_{1}=x$ and $x_{n+1}=y$. We claim that the consecutive bisectors are at least unit distance apart:
\begin{equation}\label{eq:unit-distance} 
d(\textup{Bis}(x_{i}, x_{i-1}), \textup{Bis}(x_{i}, x_{i+1}))\ge 1. 
\end{equation}

 If the closures in $\bar{X}$ of the bisectors $\textup{Bis}(x_{i}, x_{i+1})$ and $\textup{Bis}(x_{i+1}, x_{i+2})$ intersect  each other, then we have a quadrilateral $[ABCD]$ with $ \angle DAB=\angle DCB=\pi /2$ as in Figure 5(a), where $B\in \bar{X}$. Connecting $D, B$ by a geodesic segment (or a ray), we get two right triangles $[ADB]$ and $[BCD]$, and one of the angles $\angle ADB, \angle CDB$ is  $\ge \theta /2$. Without loss of generality, we can assume that $\angle ADB \geq \theta /2$. By Corollary \ref{lemma 2.2} and Remark \ref{boundary inequality 1}, $ \cosh (d(A, D))\sin \angle ADB \leq 1$.  However, we know that 
 $$
\cosh (d(A, D))\sin (\angle ADB)\geq \cosh(L/2) \sin (\theta/2)>1,
$$
 which is a contradiction. Thus, the closures of $\textup{Bis}(x_{i}, x_{i+1})$ and $\textup{Bis}(x_{i+1}, x_{i+2})$ are disjoint.

\begin {figure}[H]
\centering
\includegraphics[width=6.00in]{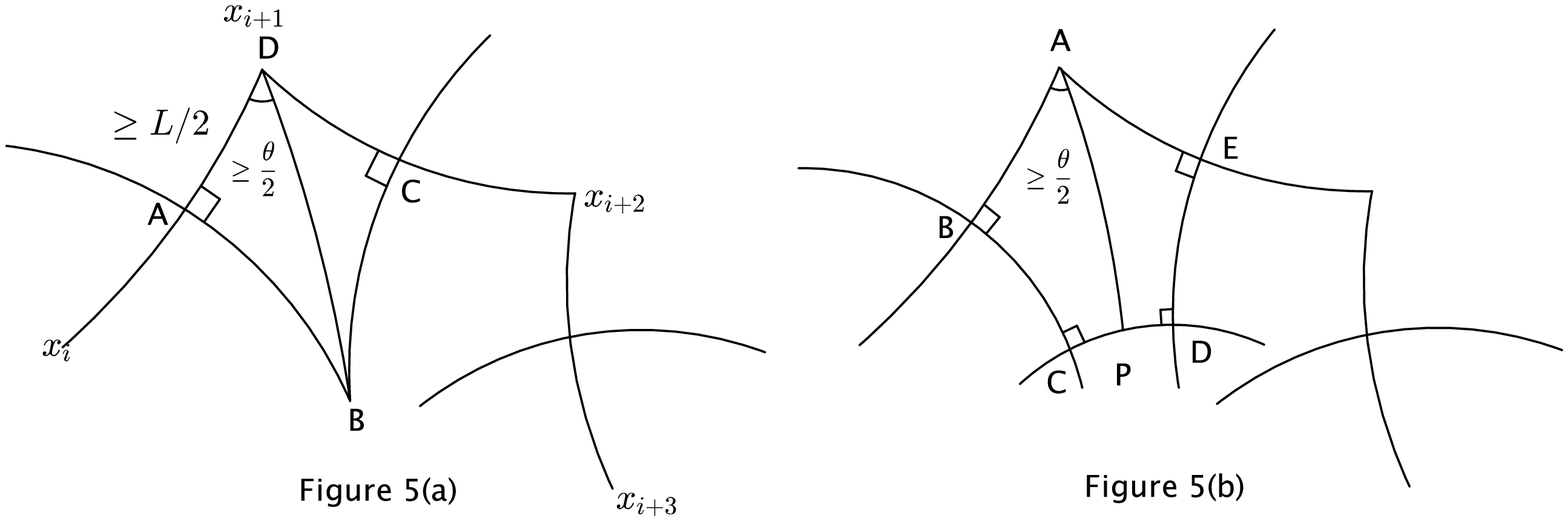}

\caption{}
\end{figure}

Let $C\in \textup{Bis}(x_{i}, x_{i+1}), D\in \textup{Bis}(x_{i+1}, x_{i+2})$ denote points (not necessarily unique) such that $d(C,D)$ is  the minimal distance between these perpendicular bisectors. Since  $CB\subset 
\textup{Bis}(x_{i}, x_{i+1})$, $DE\subset \textup{Bis}(x_{i+1}, x_{i+2})$, it follows that the segment $CD$ is orthogonal to both $CB$ and $DE$. The segment $CD$ lies on a unique (up to reparameterization) bi-infinite geodesic $\xi\eta$. Then $A\in \mathscr{N}_{P}(\xi, \eta)$ for some point $P\in \xi\eta$. We claim that $P\in CD$. Otherwise, we  obtain a triangle in $X$ with two right angles, which is a contradiction. Hence, the geodesic $AP\subseteq \mathscr{N}_{P}(C, D)$ and $AP$ is orthogonal to $CD$ as in Figure 5(b). We get two quadrilaterals $[ABCP]$ and $[APDE]$. Without loss of generality, assume that $\angle BAP\geq \theta/2$. By Corollary \ref{lemma 2.5}, 
$$\cosh(d(C, D))\geq \cosh(d(C,P)) \geq \cosh(L/2)\sin (\theta/2)  > 2 > \cosh(1). $$
Hence, $d(C, D)> 1$. This implies the inequality \eqref{eq:unit-distance}.

We now prove that the  path $\gamma$ is quasigeodesic. For each $i$, if $ d(x_{i}, x_{i+1})\geq 2L$, take the point $y_{i1}\in \gamma_{i}$ such that $d(x_{i}, y_{i1})=L$. If $L\leq d(y_{i1}, x_{i+1})< 2L$, we stop. Otherwise, take the  point $y_{i2}\in \gamma_{i}$ such that $d(y_{i1}, y_{i2})=L$. If $d(y_{i2}, x_{i+1})\geq 2L$, we continue the process until we get $y_{ij}$ such that $L\leq d(y_{ij}, x_{i+1})< 2L$. Thus we get a new partition of the piecewise geodesic path $\gamma$:
$$\gamma=\gamma '_{1}\ast \cdots \ast \gamma '_{n'}$$
 such that for each $i$, $L\le length(\gamma'_i)< 2L$, and consecutive 
 geodesic arcs $\gamma '_{i}$ and $\gamma '_{i+1}$ meet either at the angle $\pi$ or, at least, at the angle $\ge \theta$. See Figure 6(a).

In order to prove that $\gamma$ is $(\lambda,\epsilon)$-quasigeodesic, with $\la\ge 1$ and $\epsilon\ge 0$, we need to verify the inequality
$$
\dfrac{1}{\lambda} length (\gamma |_{[t_{a},t_{b}]})-\epsilon \leq d(a,b) \leq \lambda \cdot length (\gamma |_{[t_{a},t_{b}]})+\epsilon$$
for all pairs of points $a, b\in \gamma$, where $\gamma(t_{a})=a$ and $\gamma(t_{b})=b$. 
The upper bound (for arbitrary $\la\ge 1$ and $\epsilon\ge 0$) 
follows from the triangle inequality and we only need to establish the lower bound. 

The main case to consider is when  $a, b $ are both terminal endpoints of some geodesic pieces  
$\gamma '_{i}, \gamma '_{j}$ of $\ga$; see Figure 6(b). 
The  bisectors of the geodesic segments of $\gamma$ divide $ab$ into several pieces, 
and, by \eqref{eq:unit-distance}, each piece has  length  $\ge 1$. At the same time, each arc $\ga_k'$ of $\ga$ has length $< 2L$. 
Thus, $d(a,b)\ge |j-i|$, while 
$$
2L |j-i|+2L> length (\gamma |_{[t_{a}, t_{b}]}).$$
We obtain:
$$
d(a, b)\ge \frac{1}{2L} length (\gamma |_{[t_{a}, t_{b}]})-1.
$$
Lastly, general points $a\in \ga'_i, b\in \ga'_j$ are within distance $< 2L$ from the terminal endpoints $a', b'$ of these segments. Hence,
\begin{multline*}
d(a, b)\ge d(a', b') - 4L \ge \frac{1}{2L} length (\gamma |_{[t_{a'}, t_{b'}]})-1 -4L \ge
\frac{1}{2L} length (\gamma |_{[t_{a}, t_{b}]})-1 -4L\\ 
=\frac{1}{2L} length (\gamma |_{[t_{a}, t_{b}]}) - (4L +1).  
\end{multline*} 
Therefore, $\ga$ is a $(2L, 4L+1)$-quasigeodesic. 
\end{proof}

\begin {figure}[H]
\centering
\includegraphics[width=5.5in]{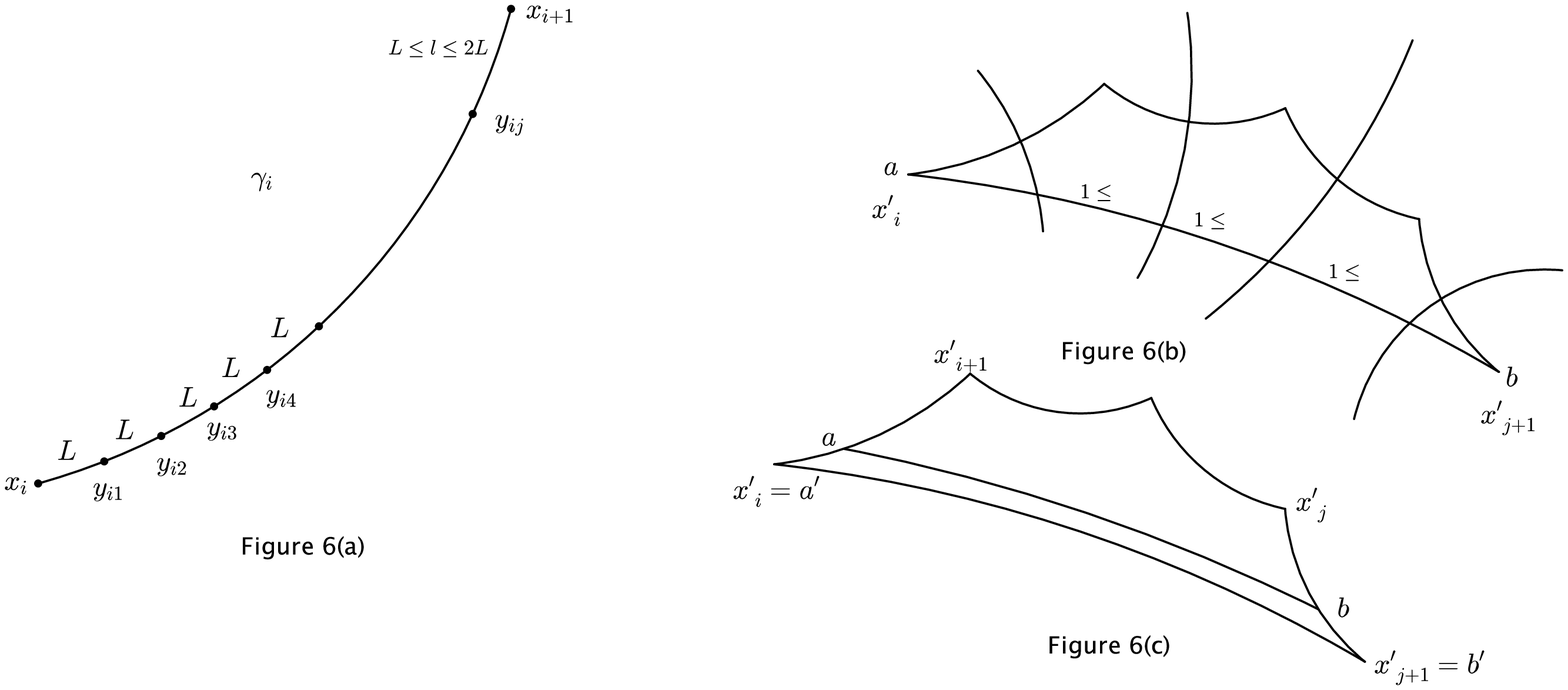}
\caption{}
\end{figure}

\begin{proposition}[Piecewise-geodesic paths with long and short edges] 
\label{qua}
Define the function 
$$L(\theta, \varepsilon)= 2\cosh^{-1}\left( \dfrac{ e^{2}+1}{2\sin (\alpha/2)}\right)+1$$
where $\alpha=\min\{ \theta, \pi/2-\arcsin(1/ \cosh \varepsilon)\}.$

Suppose that $\gamma=\gamma_{1}\ast \cdots \ast \gamma_{n}\subseteq \bar{X}$ is a piecewise geodesic path from $x$ to $y$ such that:
\begin{enumerate}
\item Each geodesic arc $\gamma_{j}$ has length either at least $\varepsilon>0$ or at least $L=L(\theta, \varepsilon)$. 

\item 
If $\gamma_{j}$ has length $<L$, then the adjacent geodesic arcs $\gamma_{j-1}$ and $\gamma_{j+1}$ have lengths at least $L$ and  $\gamma_{j}$ meets $\gamma_{j-1}$ and $\gamma_{j+1}$ at angles 
 $\geq \pi /2$.

\item
Other adjacent geodesic arcs meet at an angle $\geq \theta$. 

\end{enumerate}
Then $\gamma$ is a (2L, 4L+3)-quasigeodesic. 
\end{proposition}

\begin{proof}
We call an arc $\gamma_j$ \emph{long} if its length is $\geq L$ and \emph{short} otherwise. Notice that $\gamma$ contains no consecutive short arcs. Unlike the proof of Proposition \ref{piecewise geodesic path}, we cannot claim that the bisectors of 
consecutive arcs of $\ga$ are unit distance apart (or even disjoint). Observe, however, that by 
the same proof as in Proposition \ref{piecewise geodesic path}, the bisectors of 
 every consecutive pair $\ga_j, \ga_{j+1}$ of {\em long} arcs are 
at least unit distance apart. 

Consider, therefore, short arcs. Suppose that $\gamma_j=x_{j}x_{j+1}$ is a short arc. Then $\gamma_{j-1}=x_{j-1}x_{j}$ and $\gamma_{j+1}=x_{j+1}x_{j+2}$ are long arcs. Consider the geodesic $x_{j-1}x_{j+1}$ and the triangle $[x_{j-1}x_{j}x_{j+1}]$. By Proposition \ref{lemma 2.3}, $d(x_{j-1}, x_{j+1})\geq d(x_{j-1}, x_{j})\geq L$ and by Corollary \ref{lemma 2.2}, 
$$\cosh \varepsilon \sin \angle x_{j}x_{j+1}x_{j-1}\leq \cosh (d(x_{j}, x_{j+1}))\sin \angle x_{j}x_{j+1}x_{j-1}\leq 1. $$
Hence, 
$$\angle x_{j}x_{j+1}x_{j-1}\leq\arcsin\left(\dfrac{1}{\cosh \varepsilon}\right), $$
and 
$$\angle x_{j-1}x_{j+1}x_{j+2}\geq \dfrac{\pi}{2}-\arcsin\left(\dfrac{1}{\cosh \varepsilon}\right).$$
By a similar argument to the one of Proposition \ref{piecewise geodesic path}, the bisectors of the arcs $x_{j-1}x_{j+1}$ and $x_{j+1}x_{j+2}$ are at least distance 2 apart, see Figure \ref{general piecewise}. 

\begin {figure}[H]
\centering
\includegraphics[width=2.5in]{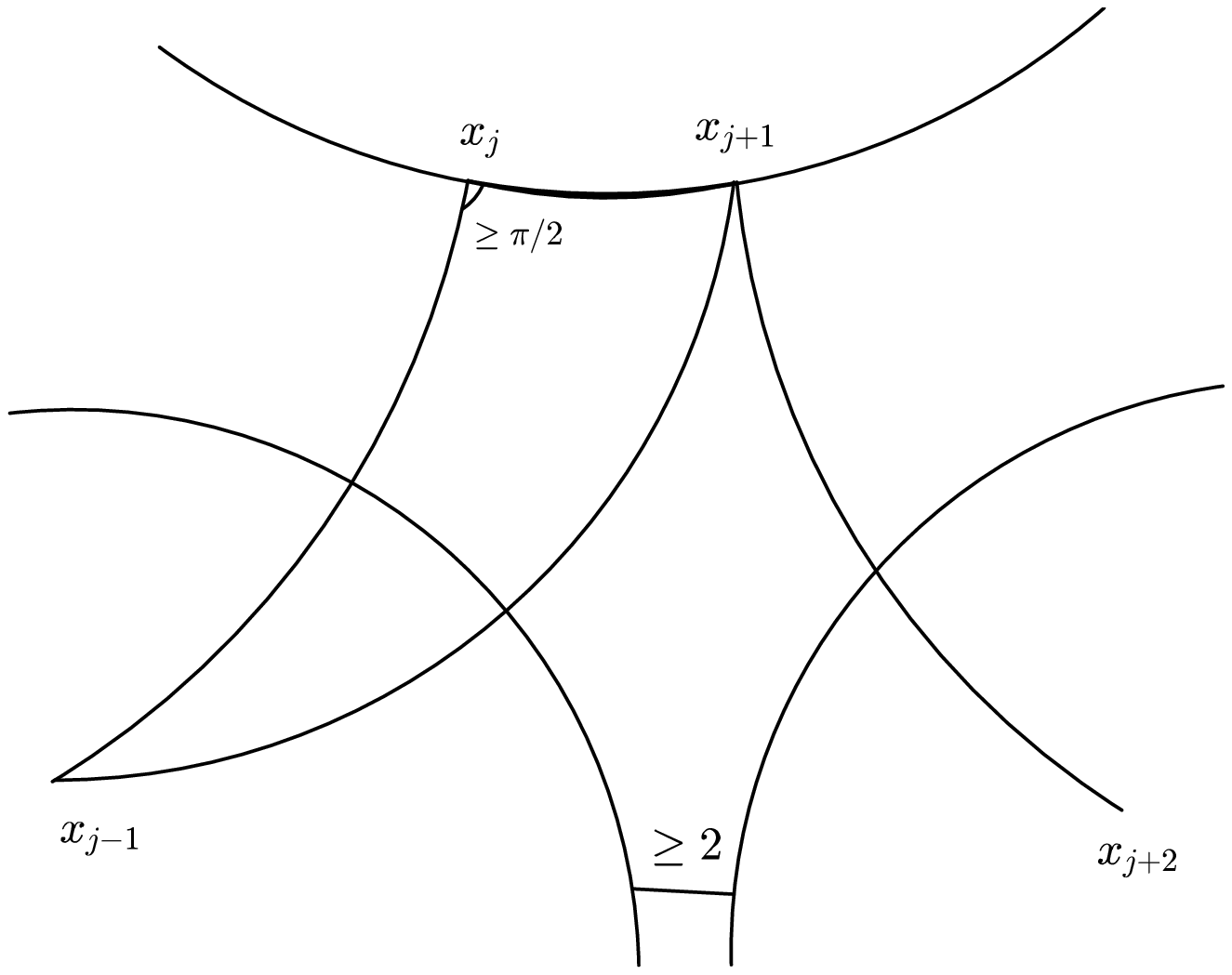}

\caption{ \label{general piecewise}}
\end{figure}

\medskip 
We now prove that $\ga$ is a $(2L, 4L+3)$-quasigeodesic. By the same argument as in Proposition \ref{piecewise geodesic path}, we can assume that all  
 long arcs of $\ga$ are  shorter than $2L$ (and short arcs, are, of course, shorter than $L$).

 As in the proof of Proposition \ref{piecewise geodesic path}, we first 
 suppose that points $a=\gamma(t_{a}), b=\gamma({t_{b}})$ in $\gamma$  are terminal points of  arcs $\ga_i, \ga_j$, $i<j$. Consider  bisectors of $x_{k-1}x_{k+1}$ for short arcs $\gamma_{k}$ in  $\gamma \mid_{[t_{a}, t_{b}]}$ and bisectors of the remaining long arcs except $\gamma_{k-1}$.  They  divide $ab$  into several segments, each of which has length at least 2. By adding these lengths together, we obtain the inequality 
 $$
 d(a, b)\ge j - i- 2,
 $$
 while 
 $$
2(j-i+1)L \ge length (\gamma \mid_{[t_{a}, t_{b}]}).$$
Putting these inequalities together, we obtain
$$
d(a, b)\ge \frac{1}{2L} length (\gamma \mid_{[t_{a}, t_{b}]}) -3.  
$$ 
Lastly, for general points $a, b$ in $\ga$, choosing $a', b'$ 
as in the proof of Proposition \ref{piecewise geodesic path}, we get:  
\begin{multline*}
d(a, b)\ge d(a', b') - 4L \ge \frac{1}{2L} length (\gamma |_{[t_{a'}, t_{b'}]}) -4L -3 \ge
\frac{1}{2L} length (\gamma |_{[t_{a}, t_{b}]})  -4L -3 \\ 
=\frac{1}{2L} length (\gamma |_{[t_{a}, t_{b}]}) - (4L +3).  
\end{multline*} 
Thus, $\gamma$ is a $(2L, 4L+3)$-quasigeodesic. 
\end{proof}

\begin{remark}
By the Morse Lemma, the Hausdorff distance between the quasigeodesic path $\gamma$ and $xy$ is at most  $C=C(L)$, \cite[Lemma 9.38, Lemma 9.80]{KD}.
\end{remark}

\section{Loxodromic products}
\label{sec:loxodromic}

In order to prove our  generalization of  Bonahon's theorem for torsion-free groups, we need to  construct a loxodromic element with  uniformly bounded word length in $\langle f, g \rangle$ where $f, g$ are two parabolic isometries  generating a discrete nonelementary subgroup of $\Isom(X)$. To deal with the case of general  discrete subgroups, possibly containing elliptic elements, 
we also need to extend this result to pairs of  elliptic isometries $g_1, g_2$.

\medskip
We first consider discrete subgroups generated by parabolic isometries. Our goal is to prove Theorem \ref{proposition 3.16}. For the proof of this theorem we will need several technical results.

\begin{lemma}\cite[Theorem $\Sigma_{m}$]{LM}
\label{nonempty intersection}
Let $F=\lbrace A_{1}, A_{2}, \cdots, A_{m} \rbrace$ be a family of open subsets of an $n$-dimensional  topological space $X$. If for every subfamily $F'$ of size $j$ where $1 \leq j \leq n+2$, the intersection $\cap F'$ is nonempty and contractible, then the intersection $\cap F$ is nonempty. 
\end{lemma}

\begin{proof}
This lemma is a special case of the topological Helly theorem \cite{LM}. Here we give another proof of the lemma. 
Suppose $k$ is the smallest integer such that there exists a subfamily $F'=\lbrace A_{i(1)}, A_{i(2)}, \cdots, A_{i(k)} \rbrace$ of size $k$ with empty  intersection $\cap F'= \emptyset$. By the assumption, $k\geq n+3$.  Then 
$$
U:=\bigcup_{1\le j\le k} A_{i(j)}
$$ 
is homotopy equivalent to the nerve $N(F')$ \cite[Corollary 4G.3]{AT}, which, in turn, is homotopy equivalent to $S^{k-2}$. Then $H_{k-2}(S^{k-2})\cong H_{k-2}(U)\cong \Z$, which is a contradiction since $k-2\geq n+1$ and $X$ has dimension $n$. 

\end{proof}

\begin{proposition}
\label{Helly}
Let $X$ be a $\delta$-hyperbolic $n$-dimensional Hadamard space. 
Suppose that $B_{1},  \cdots,B_{k}$ are convex subsets of $X$ such that $B_{i}\cap B_{j}\neq \emptyset$ for all $i$ and $j$. Then there is a point $x\in X$ such that $d(x, B_{i})\leq n \delta$ for all $i=1,...,k$. 
\end{proposition}

\begin{proof}
For $k=1, 2$, the lemma is clearly true. 

We first claim that for each $3 \leq k\leq n+2$, there exists a point $x\in X$ such that $d(x, B_{i})\leq (k-2)\delta$. We prove the claim by induction on $k$. When $k=3$, pick points $x_{ij}\in B_i\cap B_j$, $i\ne j$. Then $x_{ij} x_{il}\subset B_i$ for all $i, j, l$. Since $X$ is $\delta$-hyperbolic, there exists a point $x\in X$ within distance $\le \delta$ from all three sides of the geodesic triangle $\left[ x_{12} x_{23} x_{31} \right] $. Hence,
 $$
d(x, B_{i})\leq \delta, i=1,2,3$$
as well. 

Assume that the claim holds for $k-1$.  Set $B'_{i}= \bar{N}_{\delta}(B_{i})$ and $C_{i}=B'_{i}\cap B_{1}$ where $i\in \lbrace 2, 3, \cdots, k \rbrace$. By the convexity of the 
distance function on $X$, each $B'_{i}$ is still convex in $X$ and, hence, is a Hadamard space. Furthermore, each $B_i'$ is again $\delta$-hyperbolic. 

We claim that $C_{i}\cap C_{j}\neq \emptyset$ for all $i, j \in \lbrace 2, 3, \cdots, k \rbrace $.  
By the nonemptyness assumption, there exist points $x_{1i}\in B_{1}\cap B_{i}\neq \emptyset, x_{1j}\in B_{1}\cap B_{j}\neq \emptyset $ and $x_{ij}\in B_{i}\cap B_{j}\neq \emptyset$. 
By $\delta$-hyperbolicity of $X$, there exists a point $y\in x_{1i}x_{1j}$ such that $d(y, x_{1i}x_{ij})\leq \delta$, 
$d(y, x_{2j} x_{ij})\leq \delta $. 

Therefore, $y\in B_1\cap \bar{N}_\delta(B_i) \cap \bar{N}_\delta(B_j)= C_i\cap C_j$. By the  induction hypothesis, 
there exists a point $x'\in X$ such that $d(x', C_{i})\leq (k-3)\delta$ for each $i\in \lbrace 2, 3, \cdots, k \rbrace$. 
Thus,
$$
d(x', B_{i})\leq (k-2)\delta, i\in \lbrace 1, 2, \cdots, k \rbrace$$
as required.

 For $k>n+2$, set $U_{i}=\bar{N}_{n \delta}(B_{i})$. Then by the claim, we know that for any subfamily of $\lbrace U_{i} \rbrace$ of size $j$ where $1 \leq j \leq n+2$, its intersection is nonempty and the intersection is contractible since it is convex. By Lemma \ref{nonempty intersection}, the intersection of the family $\lbrace U_{i} \rbrace$ is also nonempty. Let $x$ be a point in this intersection. Then $d(x, B_{i})\leq n \delta$ for all $i \in \lbrace 1, 2, \cdots, k \rbrace$. 

\end{proof}

\begin{proposition}
\label{dis}
There exists a function ${\mathfrak k}: \R_+\times \R_+\to \N$ with the following property. 
Let $g_{1}, g_{2}, \cdots, g_{k}$ be parabolic elements in a discrete subgroup $\Gamma<\Isom(X)$. For each $g_i$ let  
 $G_{i}< \Gamma$ be the unique maximal parabolic subgroup containing $g_{i}$, i.e. $G_i=\stab_\Ga(p_i)$, where $p_i\in \geo X$ is the fixed point of $g_i$. Suppose that 
$$
T_{\varepsilon}(G_{i})\cap T_{\varepsilon}(G_{j})=\emptyset
$$
for all $i\ne j$. Then, whenever $k\ge {\mathfrak k}(D,\varepsilon)$, there exists a pair of indices ${i}, {j}$ with 
$$d(T_{\varepsilon}(G_{i}), T_{\varepsilon}(G_{j})) > D.$$ 
\end{proposition}

\begin{proof} For each $i$, 
 $\Hull(T_{\varepsilon}(G_{i}))$ is convex and by Remark \ref{constant for quasiconvex},  $\Hull(T_{\varepsilon}(G_{i}))\subseteq \bar{N}_{r}(T_{\varepsilon}(G_{i}))$, for some uniform constant $r=r_{\kappa}(\delta)$.  Suppose that $g_{1}, g_{2}, \cdots, g_{k}$ and $D$ are such that for all 
$i$ and $j$, 
$$
d(T_{\varepsilon} (G_{i}), T_{\varepsilon}(G_{j})) \leq D.$$
Then $d(\Hull(T_{\varepsilon}(G_{i})), \Hull(T_{\varepsilon}(G_{j}))) \leq D$.

Our goal is to get a uniform upper bound on $k$. Consider the $D/2$-neighborhoods 
$\bar{N}_{D/2}(\Hull(T_{\varepsilon}(G_{i})))$. They are convex in $X$ and have nonempty pairwise intersections. 
Thus, by Proposition \ref{Helly}, there is a point $x\in X$ such that 
$$
d(x, T_{\varepsilon}(G_{i}))\leq R_1:=n \delta + \frac{D}{2} + r, i=1,...,k. 
$$ 
Then 
$$
T_{\varepsilon}(G_{i}) \cap B(x, R_1)\neq \emptyset, i=1,...,k. $$

Next, we claim that there exists  $R_{2}\geq 0$, depending only on $\varepsilon$,  
such that 
$$
T_{\varepsilon}(G_{i}) \subseteq \bar{N}_{R_{2}}(T_{\varepsilon /3}(G_{i})).$$

Choose any point $y \in T_{\varepsilon}(G_{i})$ and let $\rho_i: [0,\infty)\to X$ be the ray $yp_i$. 
By Lemma \ref{prop 1.1},  there exists  $R=R(\varepsilon)$ such that 
$$
d(\rho_i(t), g( \rho_i(t)))\le R e^{-t}
$$
whenever $g\in G_{i}$ is a parabolic (or elliptic) isometry such that 
$$
d(y, g(y))\leq \varepsilon.$$

 Let $t=\textup{max} \lbrace \ln (3R/\varepsilon), 0 \rbrace$. Then $d(\rho_{i}(t), g(\rho_{i}(t)))\leq \varepsilon /3$ and, therefore, 
 $$
 T_{\varepsilon}(G_{i})\subseteq \bar{N}_{t}(T_{\varepsilon /3}(G_{i}))$$ 
 for all $i$. Let  $R_{2}=t$. By the argument above, $B(x, R_{1}+R_{2})\cap T_{\varepsilon /3}(G_{i}) \neq \emptyset$ for all $i$. Assume that $z_{i}\in B(x, R_{1}+R_{2})\cap T_{\varepsilon /3}(G_{i})$. Then $B(z_{i}, \varepsilon /3)\subseteq B(x, R_{3})$ where $R_{3}=R_{1}+R_{2}+\varepsilon /3$. By Lemma \ref{smaller cusp}, 
 $B(z_{i}, \varepsilon /3)\subseteq B(x, R_{3})\cap T_{\varepsilon}(G_{i})$. Since $T_{\varepsilon}(G_{i})$ and $T_{\varepsilon}(G_{j})$ are disjoint for all $i\ne j$, the metric balls 
 $B(z_{i}, \varepsilon /3)$ and $B(z_{j}, \varepsilon /3)$ are also disjoint. Recall that  $V(r, n)$ denotes the volume of the  $r$-ball in $\H^{n}$.  Then Lemma \ref{lem:packing} implies that 
  for every  
 $$
 k\geq {\mathfrak k}(D, \varepsilon):= \frac{C_{n}e^{\kappa(n-1)R_{3}}}{V(\varepsilon/3, n)}  +1, 
 $$
 there exist $i$, $j$, $1\le i, j\le k$, such that $d(T_{\varepsilon}(G_{i}), T_{\varepsilon}(G_{j}))> D$. 
\end{proof}

\begin{figure}[H]
\centering
\includegraphics[width=1.5in]{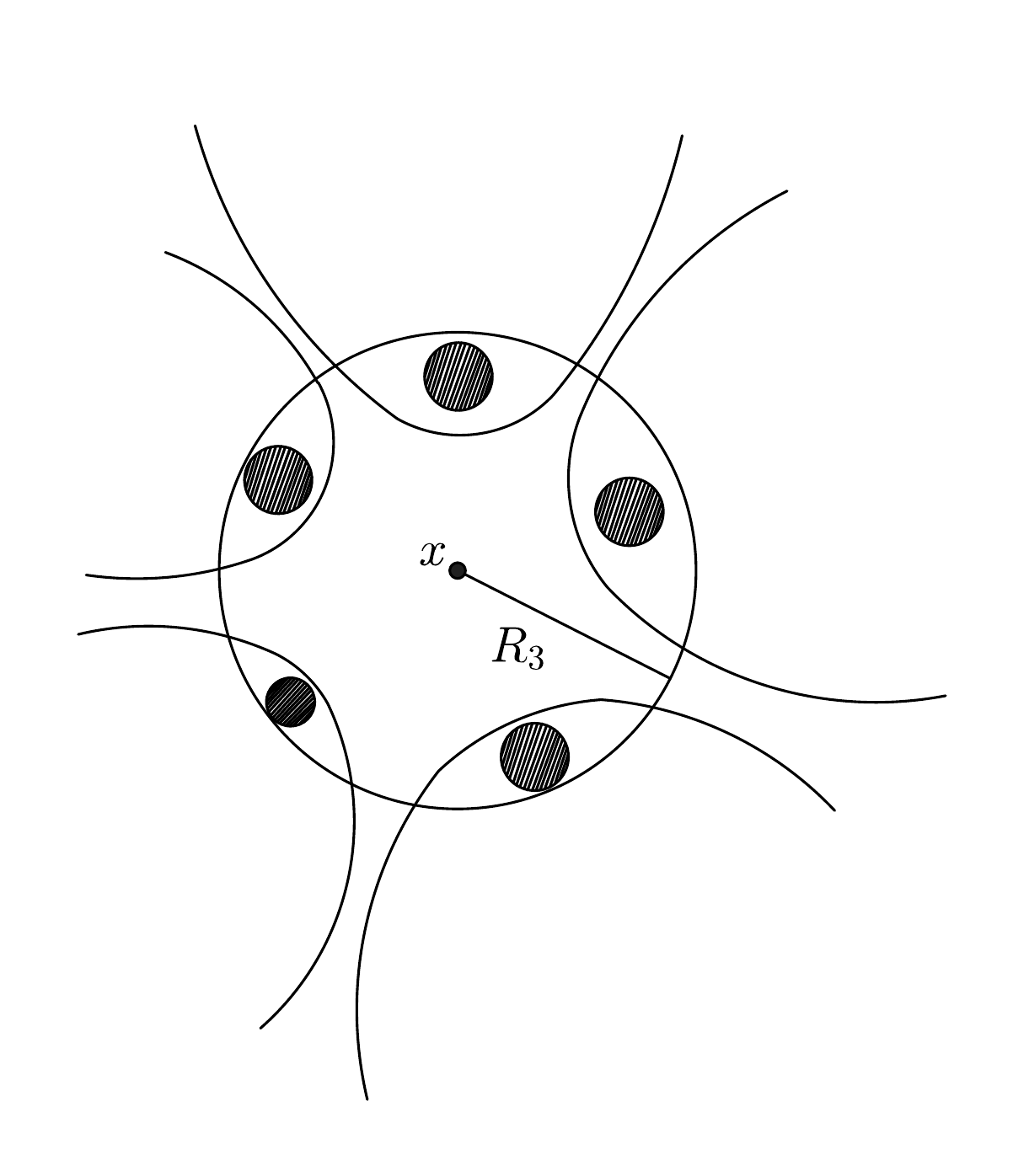} 
\caption{}
\end{figure}

\begin{proposition}
\label{lox}
Suppose that  $g_{1}, g_{2}$ are  parabolic isometries of $X$. There exists a constant $L$ which only depends on $\varepsilon$ such that if $d(Mar(g_1, \varepsilon), Mar(g_2, \varepsilon))> L$, then $h=g_{2}g_{1}$ is loxodromic. 
\end{proposition}

 \begin{figure}[H]
\centering
\includegraphics[width=5.5in]{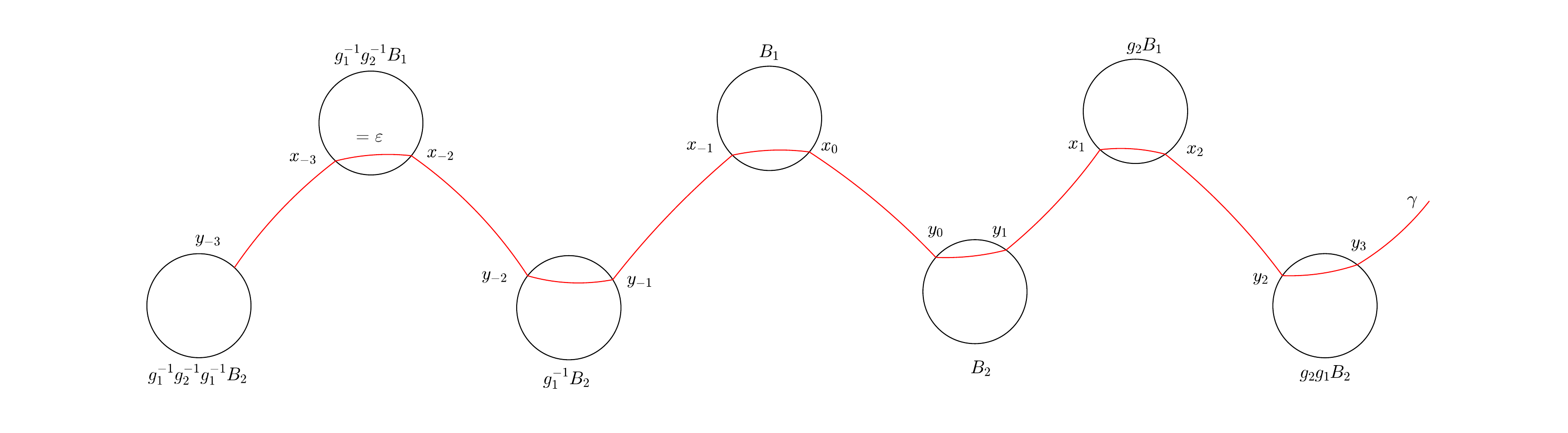} 
\caption{   \label{parabolicpic}}
\end{figure}

\begin{proof}
Let  $B_{i}=Mar(g_i, \varepsilon)$, so $d(B_{1}, B_{2})> L$. Consider the orbits of $B_{1}$ and $B_{2}$ under the action of the cyclic group generated by $g_{2}g_{1}$ as in Figure  \ref{parabolicpic}. Let $x_{0}\in B_{1}, y_{0}\in B_{2}$ denote points such that $d(x_{0}, y_{0})$ minimizes the distance function between points of $B_{1}$ and $B_{2}$. For positive integers $m> 0$, we let $$x_{2m-1}=(g_{2}g_{1})^{m-1}g_{2}(x_{0}), \quad x_{2m}=(g_{2}g_{1})^{m}(x_{0})$$ and $$y_{2m-1}=(g_{2}g_{1})^{m-1}g_{2}(y_{0}), \quad y_{2m}=(g_{2}g_{1})^{m}(y_{0}).$$ Similarly, for negative integers $m<0$, we let $$x_{2m+1}=(g_{2}g_{1})^{m+1}g_{1}^{-1}(x_{0}), \quad x_{2m}=(g_{2}g_{1})^{m}(x_{0})$$ and $$y_{2m+1}=(g_{2}g_{1})^{m+1}g_{1}^{-1}(y_{0}), \quad y_{2m}=(g_{2}g_{1})^{m}(y_{0}).$$ 

We construct a sequence of piecewise geodesic paths $\lbrace \gamma_{m} \rbrace$ where 
$$
\gamma_{m}=x_{-2m}y_{-2m}\ast y_{-2m}y_{-2m+1}\cdots \ast x_{0}y_{0} \ast y_{0}y_{1}\ast y_{1}x_{1} \cdots \ast x_{2m}y_{2m}
$$
 for positive integers $m$. Observe that $d(x_{i}, y_{i})=d(B_{1}, B_{2})>L$  and $d(x_{2i-1}, x_{2i})= \varepsilon$,  $d(y_{2i}, y_{2i+1})= \varepsilon $    for any integer $i$. By convexity of $B_1, B_2$, the angle between any adjacent geodesic arcs in $\gamma_{m}$ is at least $\pi /2$. Let $\gamma$ denote the limit of the sequence $( \gamma_{m} )$.  By Proposition \ref{qua}, there exists a constant  $L>0$ such that the piecewise geodesic path $\gamma : \mathbb{R} \rightarrow X$ is unbounded and is a uniform quasigeodesic invariant under the action of $h$. By the Morse Lemma \cite[Lemma 9.38, Lemma 9.80]{KD}, the Hausdorff distance between $\gamma$ and the complete geodesic which connects the endpoints of $\gamma$ is bounded by a uniformly constant $C$. Thus, $g_{2}g_{1}$ fixes the endpoints of $\gamma$ and acts on the complete geodesic as a translation. We conclude that $g_{2}g_{1}$ is loxodromic. \end{proof}

\begin{theorem}
\label{proposition 3.16}
Suppose that $g_{1}, g_{2}$ are two parabolic elements  with different fixed points. Then there exists a word $w\in \langle g_{1}, g_{2} \rangle$ such that $|w|\leq 4{\mathfrak k}(L, \varepsilon)+2$ and $w$ is loxodromic where $|w|$ denotes the length of the word and ${\mathfrak k}(L, \varepsilon)$ is the function in Proposition \ref{dis}, $0<\varepsilon\leq \varepsilon(n, \kappa)$ and $L$ is the constant in Proposition \ref{lox}. 
\end{theorem}

\begin{proof}
Let $p_{i}\in \geo X$ denote the fixed point of the parabolic isometry $g_{i}$,  $i=1, 2$.

Assume that every element in $\langle g_{1}, g_{2} \rangle$ of word length at most $2k(L, \varepsilon)+1$ is parabolic (otherwise, there exists a loxodromic element $w\in \langle g_{1}, g_{2} \rangle $ of word-length $\leq 4\mathfrak{k}(L, \varepsilon)+2$).

Consider the parabolic elements 
$g_{2}^{i}g_{1}g_{2}^{-i}\in \langle g_{1}, g_{2} \rangle$, $0 \leq i \leq {\mathfrak k}(L, \varepsilon)$. The fixed point (in $\geo X$) of each $g_{2}^{i}g_{1}g_{2}^{-i}$ is $g_{2}^{i}(p_{1})$. We claim that the points $g_{2}^{i}(p_{1})$ and 
$g_{2}^{j}(p_{1})$ are distinct for $i\ne j$. If not, $g_{2}^{i}(p_{1})=g_{2}^{j}(p_{1})$ for some $i>j$. Then $g_{2}^{i-j}(p_{1})=p_{1}$, and, thus,  $g_{2}^{i-j}$ has two distinct fixed points $p_{1}$ and $p_{2}$.  
This is a contradiction since any parabolic element has only one fixed point. Thus, $g_{2}^{i}g_{1}g_{2}^{-i}$ are parabolic elements with distinct fixed points for all $0 \leq i \leq k(L, \varepsilon)$. 
Since $0<\varepsilon\leq \varepsilon(n, \kappa)$, $T_{\varepsilon}(\langle g_{2}^{i}g_{1}g_{2}^{-i}\rangle), T_{\varepsilon}(\langle g_{2}^{j}g_{1}g_{2}^{-j}\rangle)$ are disjoint for any pair of indices $i, j$ \cite{Bo2}. 
By Proposition \ref{dis}, there exist  $0 \leq i , j \leq {\mathfrak k}(L ,\varepsilon)$ such that 
$$d(Mar(g_{2}^{i}g_{1}g_{2}^{-i}, \varepsilon), Mar(g_{2}^{j}g_{1}g_{2}^{-j}, \varepsilon))> L.$$
By Proposition \ref{lox}, the element $g_{2}^{j}g_{1}g_{2}^{i-j}g_{1}g_{2}^{-i}\in \langle g_{1}, g_{2} \rangle$ is loxodromic, and its word length is  $\le 4{\mathfrak k}(L, \varepsilon)+2$. Thus we can find a word $w\in \langle g_{1}, g_{2} \rangle$ such that $|w|\leq 4k(L, \varepsilon)+2$ and $w$ is loxodromic. 
\end{proof}

\begin{remark}
According to Lemma \ref{lem:large displacement}, for every parabolic isometry $g\in \Isom(X)$ and $x\notin T_{\varepsilon}(\langle g\rangle)$, there exists  
$i\in (0, N(\varepsilon, n, \kappa, L)]$ such that $d(x, g^{i}(x))>L$. Therefore, using an  argument similar to the one in the proof of Proposition \ref{lox}, we conclude that one of the products 
$g_1^{k_1}g_{2}^{k_2}$ 
is loxodromic, where $k_1, k_2 > 0$ are uniformly bounded from above.  This provides an alternative proof of the existence of 
loxodromic elements of uniformly bounded word length. We are grateful to the referee for suggesting this alternative argument. 
\end{remark}

\medskip
We now consider discrete subgroups generated by elliptic elements. In this setting, we will prove that every infinite discrete elementary subgroup $\Ga< \Isom(X)$ 
contains an infinite order element of uniformly bounded word-length (Lemma \ref{preserve geodesic} and Proposition \ref{produce parabolic element}).

\begin{lemma}
\label{preserve geodesic}
Suppose that the set $T=\lbrace g_{1}, g_{2}, \cdots, g_{m} \rbrace \subset \Isom(X) $ consists of elliptic elements, and the group $\langle T \rangle$ is an elementary loxodromic group.  Then there is a pair of indices $1\leq i, j \leq m$ such that $g_{i}g_{j}$ is loxodromic. 
\end{lemma}

\begin{proof}
Let $l$ denote the geodesic preserved setwise by $\langle T \rangle$. We claim that there exists $g_{i}$ which 
swaps the endpoints of $l$. Otherwise, $l$ is fixed pointwise by $\langle T \rangle$,  and $\langle T\rangle $ is a finite elementary  subgroup of $\Isom(X)$ which is a contradiction. Since $g_{i}(l)=l$, there exists $x\in l$ such that $g_{i}(x)=x$. By the same  argument as in Lemma \ref{finite elementary group}, there exists $g_{j}$ such that $g_{j}(x)\neq x$, and $g_{i}g_{j}$ is loxodromic. 
\end{proof}

For discrete parabolic elementary subgroups generated by elliptic isometries, we have the following result.

\begin{proposition}
\label{produce parabolic element}
Given $x\in X, 0<\varepsilon\leq \varepsilon(n, \kappa)$ and a discrete subgroup $\Gamma<\Isom(X)$, suppose that the set $\calF_{\varepsilon}(x) 
\subset \Gamma$ consists of elliptic elements and the group $\Gamma_{\varepsilon}(x)<\Gamma$ generated by this set is a parabolic elementary subgroup. Then there is a parabolic element $g\in \Gamma_{\varepsilon}(x)$ of  word length in $\calF_{\varepsilon}(x)$ uniformly bounded by a constant $C(n, \kappa)$.

\end{proposition}

\begin{proof}
Let  $N$  be the subgroup of $\Gamma_{\varepsilon}(x)$ generated by the set $\lbrace \gamma \in \Gamma_{\varepsilon}(x) \mid n_{\gamma}(x)\leq 0.49 \rbrace$. By Proposition \ref{nilpotent subgroup with finite index}, $N$ is a nilpotent subgroup of  $\Gamma_{\varepsilon}(x)=s_{1}N\cup s_{2}N\cdots \cup s_{I}N$ where the index $I$ is uniformly bounded and each $s_{i}$ has uniformly bounded word length $\leq m(n, \kappa)$ with respect to the generating set $\calF_{\varepsilon}(x)$ of $\Gamma_{\varepsilon}(x)$.  

Let $F=F_{S}$ denote the free group on $S=\calF_{\varepsilon}(x)$. Consider the projection map $\pi: F\rightarrow \Gamma_{\varepsilon}(x)$, and the preimage $\pi^{-1}(N)< F$. Let $T$ denote a left Schreier transversal for $\pi^{-1}(N)$ in $F$ (i.e a transverse for $\pi^{-1}(N)$ in $F$ so that every initial segment of an element of $T$ itself belongs to $T$). By the construction, every element $t\in T$ in the Schreier transversal has the minimal word length among all the elements in $t\pi^{-1}(N)$. Then the word length of $t$ is also bounded by $m(n, \kappa)$ since $t\pi^{-1}(N)=s_{i}\pi^{-1}(N)$ for some $i$. By the Reidemeister-Schreier Theorem, $\pi^{-1}(N)$ is generated by the set 
$$\mathrm{Y}=\lbrace  t\gamma_{i}s \mid t, s\in T, \gamma_{i} \in \mathcal{F}_{\varepsilon}(x), \textup{ and }  s\pi^{-1}(N)=t\gamma_{i}\pi^{-1}(N) \rbrace . $$
Since the word length of elements in a Schreier transversal is not greater than $m(n, \kappa)$, then the word length of  elements in the generating set $\mathrm{Y}$ is not greater than $2m(n, \kappa)+1$. 

Next, we claim that there exists a parabolic element in  $\pi(\mathrm{Y})$. If not, then all the elements in $\pi(\mathrm{Y})$ are elliptic. By Theorem \ref{torsion group}, all the torsion elements in $N$ form a subgroup of $N$. Hence all elements in $N=\langle \pi(\mathrm{Y}) \rangle$ are elliptic. By Lemma \ref{finite elementary group}, $N$ is finite, which contradicts  our assumption that  $\Gamma_{\varepsilon}(x)$ is infinite.  Therefore,  there exists a parabolic element in $\pi(Y)$ whose word length is $\leq 2m(n, \kappa)+1$. 
We let $C(n, \kappa)=2m(n, \kappa)+1$. 

\end{proof}

\begin{remark}
The virtually nilpotent group $\Gamma_{\varepsilon}(x)$ is uniformly finitely generated by at most $S(n, \kappa)$ isometries $\alpha$ satisfying $d(x, \alpha(x))\leq \varepsilon$, \cite[Lemma 9.4]{BaGS}. Let $F$ be the free group on the set $A$ consisting of such elements $\alpha$.  Since the number of subgroups of $F$ with a given finite index is uniformly bounded, and each subgroup has a finite free generating set 
it follows that $\pi^{-1}(N)$ has a generating set where each element has word length (with respect to $A$) uniformly bounded by some constant  $C(n, \kappa)$. Hence there is a generating set of $N$ where the word length of each element is uniformly bounded by $C(n, \kappa)$.  Similarly, there exists a parabolic element $g$ in this generating set of word length bounded by $C(n, \kappa)$ in elements  $\alpha$. 
This argument provides an alternative proof of 
 the existence of a parabolic isometry of uniformly bounded word length in  $\Gamma_{\varepsilon}(x)$.  
\end{remark}

The methods of the proof of the above results are insufficient for treating nonelementary discrete subgroups generated by elliptic elements. After proving our results we learned about the recent paper by Breuillard and Fujiwara  which can handle this case. Their theorem also implies Theorem \ref{proposition 3.16}. We decided to keep the proof of 
our theorem since it presents independent interest and is used in our subsequent paper \cite{DKL}. 

Given a finite subset $A$ of isometries of a metric space $X$, we let $A^m$ denote the subset of $\Isom(X)$ consisting of products of $\le m$ elements of $A$. 
Furthermore, define 
$$
L(A)=\inf_{x\in X}\max_{g\in A} d(x, gx).$$
If $X$ is a Hadamard space then $L(A)$ satisfies the inequality 
$$
L(A^m)\ge \frac{\sqrt{m}}{2} L(A A^{-1}),
$$
see \cite[Proposition 3.6]{BF}. If, in addition, $X$ is an $n$-dimensional  Riemannian manifold of sectional curvature bounded below by $-\kappa^2$, and the subgroup 
$\langle A \rangle < \Isom(X)$ is discrete and nonelementary, then $L(A)>  \varepsilon(n,\kappa)$, the Margulis constant of $X$. We will need the following result  
proven in \cite[Theorem 13.1]{BF}: 

\begin{theorem}
[Breuillard and Fujiwara] \label{thm:BF}
There exists an absolute constant $C>0$ such that for every $\delta$-hyperbolic space $X$ and every subset $A\subset \Isom(X)$ generating a nonelementary subgroup 
$\Gamma$ one of the following holds:

(i) $L(A)\le C\delta$. 

(ii) If $m> C$ then $\Gamma$ contains a loxodromic element of word-length $\le m$. 
\end{theorem}

This theorem implies: 

\begin{corollary}[Breuillard and Fujiwara]\label{cor:BF} 
  There exists a function $N=N(n, \kappa)$ satisfying the following. 
Suppose that $X$ is a negatively curved Hadamard manifold whose sectional curvature belongs to the interval $[-\kappa^2, -1]$. Then 
for any subset $A=A^{-1}\subset \Isom(X)$ generating a discrete nonelementary subgroup $\Gamma< \Isom(X)$, there exists a  loxodromic element of word-length $\le N$. 
\end{corollary}
\proof By the Margulis lemma, $L(A)>\varepsilon(n, \kappa)=\mu$. Moreover, $\delta=\cosh^{-1}(\sqrt{2})$ and, 
as noted above, 
$$
L(A^{k})\geq \dfrac{\sqrt{k}}{2}L(A)\geq \dfrac{\sqrt{k}}{2} \mu.$$
Therefore, by Theorem \ref{thm:BF}, for  
$$
m= N(n, \kappa) := \left\lceil (C+1) \left(\dfrac{2C\delta}{\mu}\right)^{2} \right\rceil $$
 the set $A^{m}$ contains a loxodromic element.  \qed

\section{A generalization of Bonahon's theorem}\label{sec:Bohanon}
\label{sec:generalization}

In this section, we use the construction in Section \ref{sec:loxodromic} to generalize Bonahon's theorem for any discrete subgroup $\Gamma < \textup{Isom}(X)$ where $X$ is a negatively pinched Hadamard manifold. 

\begin{lemma}
\label{lemma 3.8}
For every $\tilde{x}\in \Hull(\Lambda(\Gamma))$, 
$$\textup{hd}(\textup{QHull}(\Gamma \tilde{x}), \textup{QHull}(\Lambda (\Gamma)))< \infty$$
\end{lemma}
\proof 
By the assumption that $\tilde{x}\in \Hull(\Lambda(\Gamma))$   and Remark \ref{uniformly constant for convex}, there exists $r_{1}={\mathfrak r}_{\kappa}(2\delta)\in [0, \infty)$ such that 
$$
\textup{QHull}(\Gamma \tilde{x})\subseteq \Hull(\Lambda (\Gamma)) \subseteq \bar{N}_{r_{1}}(\textup{QHull}(\Lambda(\Gamma)))$$ 
Next, we want to prove that there exists a constant  $r_{2}\in [0, \infty)$ such that $\textup{QHull}(\Lambda(\Gamma))\subseteq \bar{N}_{r_{2}}(\textup{QHull}(\Gamma \tilde{x}))$.   

Pick any point $p\in \textup{QHull}(\Lambda(\Gamma))$.  Then $p$ lies on some geodesic $\xi \eta$ where $\xi, \eta \in \Lambda(\Gamma)$  are distinct points. Since $\xi$ and $\eta$ are in the limit set, there exist sequences of elements $(f_{i})$ and $(g_{i})$ in $\Gamma$ such that the sequence $(f_{i}(\tilde{x}))$ converges to $\xi$ and the sequence $(g_{i}(\tilde{x}))$ converges to $\eta$. By Lemma \ref{hyper}, $p\in \bar{N}_{2\delta}(f_{i}(\tilde{x})g_{i}(\tilde{x}))$ for all  sufficiently large $i$. Let $r=\textup{max} \lbrace r_{1}, 2\delta \rbrace$. Thus,  
$$
\textup{hd}(\textup{QHull}(\Gamma \tilde{x}), \textup{QHull}(\Lambda(\Gamma)))=r < \infty. \qedhere$$

\begin{remark}
Let $\gamma_{i}=f_{i}(\tilde{x}) g_{i}(\tilde{x})$. Then there exists a sequence of points 
$p_{i}\in \ga_i$, which converges to $p$. 
\end{remark}

If $\Gamma< \Isom(X)$ is geometrically infinite, then
$$
\Core(M)\cap \textup{noncusp}_{\varepsilon}(M)$$ 
is noncompact, \cite{Bo2}. By Lemma \ref{lemma 3.8}, $(\textup{QHull}(\Gamma \tilde{x})/ \Gamma) \cap \textup{noncusp}_{\varepsilon}(M)$ is unbounded. 

\medskip 
We now  generalize Bonahon's theorem to geometrically infinite discrete subgroup $\Gamma < \textup{Isom}(X)$. \\

\noindent
{\bf Proof of the implication $(1)\Rightarrow (2)$ in Theorem \ref{theo 1.3}:}
If there exists a sequence of closed geodesics $\beta_{i}\subseteq M$ whose lengths tend to $0$ as $i\rightarrow \infty$,  the sequence $( \beta_{i} )$ escapes every compact subset of $M$. From now on, we assume that there exists a constant $\epsilon>0$ which is a lower bound on  the lengths  of  closed geodesics $\beta$ in $M$.

Consider Margulis cusps  $T_{\varepsilon}(G)/G$, where $G< \Gamma$ are maximal parabolic subgroups. There exists a  constant $r\in [0, \infty), r={\mathfrak r}_{\kappa}(\delta)$ 
such that 
$$
\textup{Hull}(T_{\varepsilon}(G))\subseteq \bar{N}_{r}(T_{\varepsilon}(G))$$ 
for every maximal parabolic subgroup $G$ (see Section \ref{sec:elementary}). Let $B(G)=\bar{N}_{2+4\delta}(\textup{Hull}(T_{\varepsilon}(G)))$.  Let $M^{o}$ be the union of all subsets $B(G)/ \Gamma$ where $G$ ranges over all maximal parabolic subgroups of $\Gamma$. Further, we let $M^{c}$ denote the closure of $\textup{Core}(M)\setminus M^{o}$. Since $\Gamma$ is geometrically infinite, the noncuspidal part of the convex core, 
$$
\noncusp_{\varepsilon} (\Core(M)= \textup{Core}(M)) \setminus \textup{cusp}_{\varepsilon}(M)$$
is unbounded  by Theorem \ref{thm:gfcharacterization}. Then $M^{c}$ is also unbounded since 
$$
M^{o}\subseteq \bar{N}_{r+2+4\delta}(\textup{cusp}_{\varepsilon}(M)),$$ 

Fix a point $x\in M^{c}$ and a point $\tilde{x}\in \pi^{-1}(x)\subset X$. Let 
$$
C_{n}=B(x, nR)=\lbrace y\in M^{c}\mid d(x, y)\leq nR \rbrace,$$
 where 
 $$R=r+2+4\delta+m\varepsilon$$ 
 and  $m=C(n, \kappa)$ is the constant in Proposition \ref{produce parabolic element}.
Let $\boldsymbol{\delta} C_{n}$ denote the relative boundary 
$$
\partial C_{n} \setminus \partial M_{cusp}^{c}$$
 of $C_{n}$ where 
 $$M_{cusp}^{c}=M^{o} \cap \textup{Core}(M).$$

By Lemma \ref{lemma 3.8} $(\textup{QHull}(\Gamma \tilde{x})/ \Gamma ) \cap M^{c}$ is unbounded. For every $C_{n}$, there exists a sequence of geodesic loops $( \gamma_{i} )$ connecting $x$ to itself in $\textup{Core}(M)$ such that the Hausdorff distance hd$(\gamma_{i} \cap M^{c}, C_{n})\rightarrow \infty$ as $i\rightarrow \infty$. Let $y_{i}\in \gamma_{i}\cap M^{c}$ be such that $d(y_{i}, C_{n})$ is maximal on $\gamma_{i}\cap M^{c}$. We pick a component $\alpha_{i}$ of $\gamma_{i}\cap M^{c}$ in the complement of $C_{n}$ such that $y_{i}\in \alpha_{i}$.  Consider the sequence of geodesic arcs $( \alpha_{i} )$. 

\medskip 
After passing to a subsequence in $(\alpha_{i})$, one of the following three cases occurs: 

Case (a): Each $\alpha_{i}$ has both endpoints  $x'_{i}$ and $x''_{i}$  on $\partial M_{cusp}^{c}$ as in Figure 10(a). 
By the construction, there exist $y'_{i}$ and $y''_{i}$ in the cuspidal part such that $d(x'_{i}, y'_{i})\leq r_{1}, d(y'_{i}, y''_{i})\leq r_{1}$ where $r_{1}=2+4\delta+r$. Let $\tilde{y}'_{i}$ be a lift of $y'_{i}$ such that $\tilde{y}'_{i}\in T_{\varepsilon}(G')$ for some maximal parabolic subgroup $G'<\Gamma$. By  the definition, the subgroup $\Gamma_{\varepsilon}(\tilde{y}'_{i})$ generated by the set $$\calF_{\varepsilon}(\tilde{y}'_{i})= \lbrace \gamma \in G' \mid d(\tilde{y}'_{i}, \gamma (\tilde{y}'_{i}))\leq \varepsilon \rbrace$$ is infinite. 

We claim that there exists a parabolic element $g'\in \Gamma_{\varepsilon}(\tilde{y}'_{i})$ such that $d(\tilde{y}'_{i}, g'(\tilde{y}'_{i}))\leq m\varepsilon.$ Assume that $\calF_{\varepsilon}(\tilde{y}'_{i})=\lbrace \gamma_{1}, \cdots, \gamma_{b} \rbrace$. If $\gamma_{j}$ is parabolic for some $1\leq j \leq b$, we have $d(\tilde{y}'_{i}, \gamma_{j}(\tilde{y}'_{i}))\leq \varepsilon$. Now assume that $\gamma_{j}$ are elliptic for all $1\leq j \leq b$. By Proposition \ref{produce parabolic element}, there is a parabolic element $g'\in \Gamma_{\varepsilon}(\tilde{y}'_{i})$ of word length (in the generating set $\calF_{\varepsilon}(\tilde{y}'_{i})$) bounded by $m$. 
By the  triangle inequality, $d(\tilde{y}'_{i}, g'(\tilde{y}'_{i}))\leq m\varepsilon$.

Then we find a nontrivial  geodesic loop  $\alpha '_{i}$ contained $M^{o}$ such that $\alpha '_{i}$ connects $y'_{i}$ to itself and has length $l(\alpha '_{i})\leq m\varepsilon$. Similarly, there exists a nontrivial geodesic loop $\alpha ''_{i}$ which  connects $y''_{i}$ to itself and has length $l(\alpha ''_{i})\leq m\varepsilon$. Let $$w'=x'_{i}y'_{i} \ast \alpha '_{i} \ast y'_{i}x'_{i} \in \Omega(M, x'_{i})$$ 
and 
$$w ''=\alpha_{i}\ast x''_{i}y''_{i}\ast \alpha ''_{i}\ast y''_{i}x''_{i}\ast \alpha_{i}^{-1}\in \Omega(M, x'_{i}),$$  
where $\Omega(M, x'_{i})$ denotes the loop space of $M$.  Observe that $w'\cap C_{n-1}=\emptyset$ and $w'' \cap C_{n-1}=\emptyset$. 

Let $g', g''$ denote the elements of $\Gamma=\pi_1(M, x_i')$ represented by $w'$ and $w''$ respectively. By the construction, $g'$ and $g''$ are both parabolic.  We claim that $g'$ and $g''$ have different fixed points in $\geo X$. Otherwise, $g, g''\in G'$ where $G'<\Gamma$ is some maximal parabolic subgroup. Then $y'_{i}, y''_{i}\in T_{\varepsilon}(G')/ \Gamma$ and $x'_{i}, x''_{i}\in B(G')/ \Gamma$. Since $\textup{Hull}(T_{\varepsilon}(G'))$ is convex, $B(G')=\bar{N}_{2+4\delta}(\textup{Hull}(T_{\varepsilon}(G')))$ is also convex by  convexity of the distance function. Thus,  $x'_{i}x''_{i}\subseteq B(G')/ \Gamma$. However, $x'_{i}x''_{i}$ lies outside of $B(G') / \Gamma$ by construction, which is a contradiction. 

By Theorem \ref{proposition 3.16}, there exists a loxordomic element $\omega_{n}\in \langle g', g'' \rangle < \Gamma= \pi_1(M, x_i')$ 
with the word length uniformly bounded by a  constant $K={\mathfrak k}(\varepsilon,\kappa)$ independent of $n$. 
Let $w_{n}$ be a concatenation of $w'_{i}, w''_{i}$ and their reverses which represents $\omega_{n}$. 
Then  the number of geodesic arcs in $w_{n}$ is uniformly bounded by $5K$. The piecewise geodesic loop $w_{n}$ is  freely homotopic to a closed geodesic $w^{\ast}_{n}$ in $M$; hence,  by Proposition \ref{free homotopic}, $w^{\ast}_{n}$ is contained in some $D$-neighborhood of the loop $w_{n}$ where 
$$
D=\cosh^{-1}(\sqrt{2})\lceil \log_{2} 5K \rceil+\sinh^{-1}(2/\epsilon)+2\delta.
$$
Thus, $d(x, w^{\ast}_{n})\geq (n-1)R-D$. \\

\begin {figure}[H]
\centering
\includegraphics[width=6.0in]{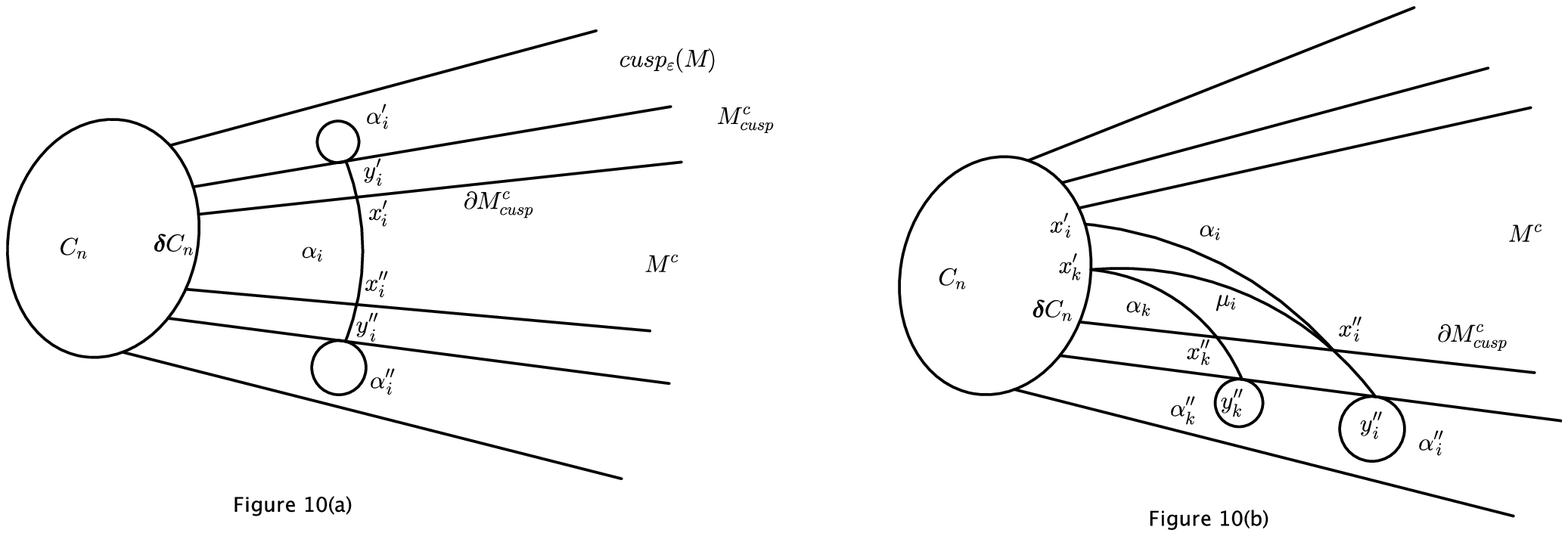}

\caption{}
\end{figure}

Case (b):  For each $i$, the geodesic arc $\alpha_{i}$ connects $x'_{i}\in \boldsymbol{\delta} C_{n}$ to $x''_{i}\in \partial M^{c}_{cusp}$, as in Figure 10(b). For each  $x''_{i}$, there exists a point $y''_{i}\in \textup{cusp}_{\varepsilon}(M)$ such that $d(x''_{i}, y''_{i})\leq r_{1}$ and a  short nontrivial geodesic loop  $\alpha ''_{i}$ contained in $M^{o}$ which connects $y''_{i}$ to itself and has length $l(\alpha ''_{i})\leq m\varepsilon$. 
Since $\boldsymbol{\delta} C_{n}$ is compact, after passing to a further subsequence in $(\alpha_i)$, there exists $k\in \mathbb{N}$ such that for all 
$i \geq k$, $d(x'_{i}, x'_{k}) \leq 1$ and less than the injectivity radius of $M$ at $x'_k$. Hence, there exists a unique shortest geodesic $x_k' x_i'$ in the manifold 
$M$. Let $\mu _{i}=x'_{k}x''_{i}$ denote the geodesic arc homotopic to the concatenation $x_k' x_i' * x_i' x_i''$ rel. $\{x_i', x_i''\}$. Then, by the $\delta$-hyperbolicity of $X$, the geodesic  $\mu _{i}=x'_{k}x''_{i}$ is contained  in the $(1+\delta)$-neighborhood of $\alpha_{i}$. 

Let 
$$w'_{k}=\alpha_{k}\ast x''_{k}y''_{k}\ast \alpha ''_{k}\ast y''_{k}x''_{k}\ast \alpha_{k}^{-1}\in \Omega(M, x'_{k})$$
and 
$$w ' _{i}=\mu_{i}\ast x''_{i}y''_{i}\ast \alpha ''_{i}\ast y''_{i}x''_{i}\ast (\mu_{i})^{-1}\in \Omega(M, x'_{k})$$ 
for all $i>k$. By the construction,  $w'_{i}\cap C_{n-1}=\emptyset$ for each $i\geq k$. 

Let $g_{i}$ denote the element of $\Gamma=\pi_1(M, x'_{k})$ represented by $w'_{i}$, $i\geq k$. Then each $g_{i}$ is parabolic. We claim that there exists a pair of indices $i, j\geq k$ such that $g_{i}$ and $g_{j}$ have distinct fixed points. Otherwise, assume that all parabolic elements $g_{i}$ have the same fixed point $p$. 
Then $x''_{i}\in B(G')/ \Gamma$ for any $i\geq k$ where $G'=\textup{Stab}_{\Gamma}(p)$. 

Since $\mu_{i}\cup \alpha_k$ is in the $(1+\delta)$-neighborhood of $M^{c}$, by the $\delta$-hyperbolicity of $X$ we have that  
$x''_{k}x''_{i}$ is in $(1+2\delta)$-neighborhood of $M^{c}$ for every $i > k$. By the definition of $M^c$, it follows that 
$$
x''_{k}x''_{i}\cap \bar{N}_{\delta}(\textup{Hull}(T_{\varepsilon}(G')))/ \Gamma=\emptyset.$$ 
By the construction, the length $l(\alpha_{i})\rightarrow \infty$ as $i\rightarrow \infty$. Hence, the length $l(\mu_{i})\rightarrow \infty$ and the length $l(x''_{k}x''_{i})\rightarrow \infty$  as $i\rightarrow \infty$. By Lemma \ref{enter cusp}, there exists points $z_{i}\in x''_{k}x''_{i}$ such that 
$z_{i}\in \bar{N}_{\delta}(T_{\varepsilon}(G'))/ \Gamma$ for sufficiently large $i$.  Therefore, 
$$
x''_{k}x''_{i}\cap \bar{N}_{\delta}(\textup{Hull}(T_{\varepsilon}(G')))/ \Gamma\neq \emptyset,$$ 
 which is a contradiction. 
 
 We conclude that for some $i, j\ge k$, the parabolic elements $g_i, g_j$ of $\Ga$ have distinct fixed points and, hence, generate a nonelementary subgroup 
 of $\Isom(X)$. By Theorem \ref{proposition 3.16}, there exists a loxodromic element $\omega_{n}\in \langle g_{i}, g_{j} \rangle$ with the word length uniformly bounded by a constant $K$. By the same  argument as  in Case (a), we obtain a closed geodesic $w^{\ast}_{n}$ 
 (representing the conjugacy class of $\omega_n$) in $M$ such that $d(x, w^{\ast}_{n})\geq (n-1)R-D$.   \\

\begin {figure}[H]
\centering
\includegraphics[width=7.0in]{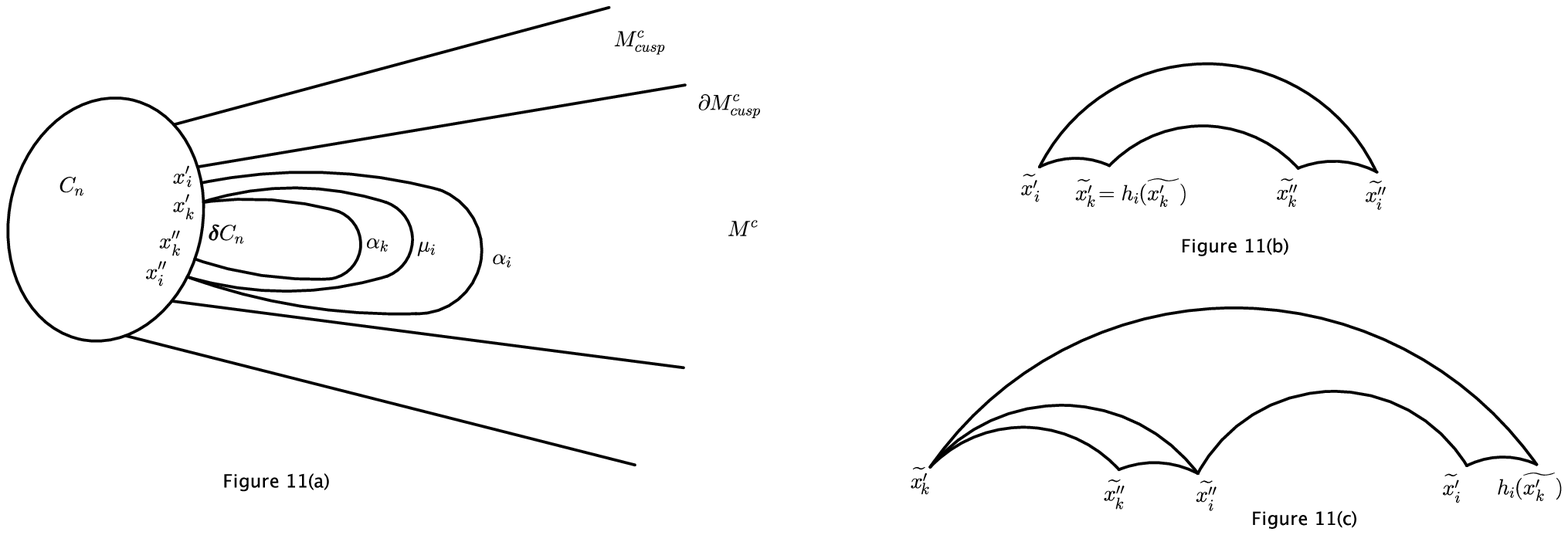}

\caption{}
\end{figure}

Case (c): We assume that for each $i$, the  geodesic arc $\alpha_{i}$ connects 
$x'_{i}\in \boldsymbol{\delta} C_n$ to $x ''_{i}\in \boldsymbol{\delta} C_{n}$. The argument is similar to the one in Case (b). 
Since $\boldsymbol{\delta} C_{n}$ is compact, after passing to a further subsequence in $(\alpha_i)$, there exists $k\in \mathbb{N}$ such that for all 
$i \geq k$, $d(x'_{i}, x'_{k})\leq 1$, $d(x''_{i}, x''_{k}) \leq 1$ and there are unique shortest geodesics $x_k' x_i'$ and $x''_{k}x''_{i}$. For each $i> k$ we define a geodesic 
$\mu_{i}=x'_{k}x''_{i}$ as in Case (b), see Figure 11(a). Then, by the $\delta$-hyperbolicity of $X$,  each $\mu_{i}$ is in the $(\delta+1)$-neighborhood of $\alpha_{i}$. 
Let $v_{i}=\alpha_{k} \ast x''_{k}x''_{i}\ast (\mu_{i})^{ -1}\in \Omega(M, x'_{k})$ for  $i> k$. By the construction, $v_{i}\cap C_{n-1}=\emptyset$. 

 Let $h_{i}$ denote the element in $\Gamma=\pi_1(M, x'_{k}) $ represented by $v_{i}$.  If $h_{i}$ is loxodromic for some $i > k$, there exists a closed geodesic $w^{\ast}_{n}$ contained in the $D$-neighborhood of $v_{i}$, cf. Case (a). In this situation, $d(x, w^{\ast}_{n})\geq (n-1)R-D$. 
 
 Assume, therefore, that $h_i$ are not loxodromic for all $i>k$.

We first claim that $h_{i}$ is not the identity for all sufficiently large $i$.  Let $\widetilde{x'_{k}}$ be a lift of $x'_{k}$ in $X$. Pick points $\widetilde{x''_{k}}, \widetilde{x''_{i}}, \widetilde{x'_{i}}$ and $h_{i}(\widetilde{x'_{k}})$ in $X$ such that $\widetilde{x'_{k}}\widetilde{x''_{k}}$ is a lift of $\alpha_{k}$, $\widetilde{x''_{k}}\widetilde{x''_{i}}$ is a lift of $x''_{k}x''_{i}$, $\widetilde{x'_{i}}\widetilde{x''_{i}}$ is a lift of $\alpha_{i}$ and $\widetilde{x'_{i}}h_{i}(\widetilde{x'_{k}})$ is a lift of $x'_{i}x'_{k}$ as in Figure 11(b) and Figure 11(c). If $h_{i}=1$, then $h_{i}(\widetilde{x'_{k}})=\widetilde{x'_{k}}$ and $d(\widetilde{x'_{i}}, \widetilde{x''_{i}})\leq 2+d(\widetilde{x'_{k}}, \widetilde{x''_{k}})$ as in Figure 11(b). By construction, the length $l(\alpha_{i})\rightarrow \infty$ as $i\rightarrow \infty$, so $d(\widetilde{x'_{i}}, \widetilde{x''_{i}})\rightarrow \infty$. Thus for sufficiently large $i$, $h_{i}(\widetilde{x'_{k}})\neq \widetilde{x'_{k}}$. 

Assume, therefore, that $h_{i}$ are not loxodromic and not the identity for all $i >k$. Then $h_{i}$ could be either parabolic  or elliptic for $i> k$.

\begin{claim*}
For every $k$, there exist  $i, j>k$ and  a loxodromic element in $\langle h_{i}, h_{j} \rangle$ whose  word length is  bounded by a constant independent of $k$. 
\end{claim*}

\begin{proof}
Suppose there is a subsequence in $(h_{i})_{i>k}$  consisting of parabolic elements. For simplicity,   we assume that $h_{i}$ are parabolic for all $i>k'$ where $k'>k$ is a sufficiently large number.  We claim that there exists a pair of indices $i, j>k'$ such that $h_{i}$ and $h_{j}$ have distinct fixed points in $\geo X$. Otherwise, all the parabolic elements $h_{i}$ have the same fixed point $p$ for $i>k'$.  By the $\delta$-hyperbolicity of $X$, 
$\widetilde{x'_{k}}h_{i}(\widetilde{x'_{k}})\subseteq \bar{N}_{3\delta+2}(\widetilde{x'_{k}}\widetilde{x''_{k}}\cup \widetilde{x''_{i}} \widetilde{x'_{i}})$. Since $\alpha_{k}$ and $\alpha_{i}$ lie outside of $B(G')/ \Gamma$ where $G'=\textup{Stab}_{\Gamma}(p)$, the segment $\widetilde{x'_{k}}h_{i}(\widetilde{x'_{k}})$ lies outside of $\bar{N}_{\delta}(\textup{Hull}(T_{\varepsilon}(G')))$.  Let $r_{3}=d(\widetilde{x'_{k}}, \textup{Hull}(T_{\varepsilon}(G')))$. Then $d(h_{i}(\widetilde{x'_{k}}), \textup{Hull}(T_{\varepsilon}(G')))=r_{3}$. 

By the construction, the length $l(\alpha_{i})\rightarrow \infty$ as $i\rightarrow \infty$. Then the length $l(\widetilde{x'_{k}} h_{i}(\widetilde{x'_{k}}))\rightarrow \infty$ as well. Observe that the points $\widetilde{x'_{k}}$ and $h_{i}(\widetilde{x'_{k}})$ lie on the boundary of $\bar{N}_{r_{3}}(\textup{Hull}(T_{\varepsilon}(G)))$ for all $i>k'$. By Lemma \ref{enter cusp}, there exist points $\widetilde{z_{i}}\in \widetilde{x'_{k}}h_{i}(\widetilde{x'_{k}})$ such that $\widetilde{z_{i}}\in \bar{N}_{\delta}(T_{\varepsilon}(G'))$ for sufficiently large $i$,  which is a contradiction. Hence, for some $i> k', j> k'$, parabolic isometries $h_{i}$ and $h_{j}$ have distinct fixed points.

By Theorem \ref{proposition 3.16}, there exists a loxodromic element $\omega_{n}\in \langle h_{i}, h_{j} \rangle$ of the word length bounded by a uniform constant $K$.

Now assume that $h_{i}$ are elliptic for all $i>k$. 

If there exist $i, j > k$ such that $\langle h_{i}, h_{j} \rangle$ is nonelementary, by Corollary \ref{cor:BF}, there exists a loxodromic element  $\omega_{n}\in \langle h_{i}, h_{j} \rangle$ of word length uniformly bounded by a constant $K$. Now suppose that $\langle h_{i}, h_{j} \rangle$ is elementary for any  pair of indices $i, j>k$. If one of the elementary subgroups is infinite and preserves  a geodesic, by Lemma \ref{preserve geodesic}, $h_{i}h_{j}$ is loxodromic. 

Assume that all the elementary subgroups  $\langle h_{i}, h_{j} \rangle$ are either finite or parabolic for  all $i, j>k$. Let $B_{i}$ denote the closure of $Mar(h_{i}, \varepsilon)$ in $\bar{X}$. If there exist $i, j$ such that $B_{i}$ and $B_{j}$ are disjoint, then $\langle h_{i}, h_{j} \rangle $ is nonelementary which contradicts our assumption.  Thus for any pair of indices $i, j>k$, $B_{i}\cap B_{j}\neq \varnothing$. There exists a uniform constant $r'$ such that $N_{r'}(B_{i})\cap N_{r'}(B_{j})\neq \varnothing$ in $X$. Hence,
 by Proposition \ref{Helly},  there exists $\tilde{z} \in X$ such that for all $i> k$ we have $d(\tilde{z}, N_{r'}(B_{i}))\leq n\delta$. For any $q \in N_{r'}(B_{i})$, $d(q, h_{i}(q))\leq 2r'+\varepsilon$ by the triangle inequality. Thus,
$$d(\tilde{z}, h_{i}(\tilde{z}))\leq 2n\delta+2r'+\varepsilon$$ 
for all $i>k$. Let $\widetilde{x}'_{k}$ denote a lift of $x'_{k}$ in $X$, and $l=d(\tilde{z}, \widetilde{x}'_{k})$. Then 
$$d(\widetilde{x}'_{k}, h_{i}(\widetilde{x}'_{k}))\leq 2l+2n\delta+2r'+\varepsilon$$
for all $i>k$. Note that $d(\widetilde{x}'_{k}, h_{i}(\widetilde{x}'_{k}))\rightarrow \infty $ as $i\rightarrow \infty$, which is a contradiction.  
\end{proof}

Thus, for some pair of indices $i, j>k$, there exists a loxodromic element $\omega_{n}\in \langle h_{i}, h_{j}\rangle$  whose word   length is uniformly bounded by some 
constant $K$.  By the same argument as in Case (a), there exists a closed geodesic $w^{\ast}_{n}$ such that $d(x, w^{\ast}_{n})\geq (n-1)R-D$. 

Thus in all cases, for each ${n}$, the orbifold $M$ contains a closed geodesic $w^{\ast}_{n}$ such that $d(x, w^{\ast}_{n})\geq (n-1)R-D$. 
The sequence of closed geodesics $\lbrace w^{\ast}_{n} \rbrace$, therefore, escapes every compact subset of $M$.  
\qed

\section{Continuum of nonconical limit points}
\label{sec:continuum}

In this section, using the generalized Bonahon theorem in Section \ref{sec:Bohanon}, for each geometrically infinite discrete subgroup $\Gamma< \Isom(X)$ 
we find a set  of nonconical limit points with the cardinality of the continuum. This set of nonconical limit points is used to prove Theorem \ref{theo 1.3}.

\begin{theorem}
\label{Theo 5.1}
If $\Gamma< \Isom(X)$ is a geometrically infinite discrete  isometry subgroup, then the set of nonconical limit points of $\Gamma$ 
has the cardinality of the continuum. 
\end{theorem}

\begin{proof}
The proof is inspired by Bishop's construction of nonconical limit points of geometrically infinite Kleinian groups in the 3-dimensional hyperbolic space $\mathbb{H}^{3}$; \cite[Theorem 1.1]{Bi}. Let $\pi: X\to M=X/\Gamma$ denote the covering projection. 
Pick a point $\tilde{x}\in X$ and set ${x}:= \pi(\tilde x)$.    If $\Gamma$ is geometrically infinite, by the generalized Bonahon theorem in Section \ref{sec:Bohanon}, 
there exists a sequence  of oriented closed geodesics $( \lambda_{i} )$ in $M$ which escapes every compact subset of $M$, i.e. 
$$
\lim_{i\to\infty} d({x}, \lambda_{i})=\infty. 
$$ 
Let $L$ be the constant as in Proposition \ref{piecewise geodesic path} when $\theta= \pi /2$. After passing to a subsequence if necessary, we can assume that $d({x}, \lambda_{1})\geq L $ and the minimal distance between any consecutive pair of geodesics $\lambda_i, \lambda_{i+1}$ is at least $L$.  For each $i$, let $l_{i}$ denote the length of the closed geodesic $\lambda_{i}$ 
and let $m_{i}$ be a positive integer such that $m_{i}l_{i}>L$.

We then pass to a subsequence in $(\lambda_i)$ as in Lemma   \ref{escape compact set} (retaining the notation $(\lambda_i)$ for the subsequence),  
so that there exists a sequence of geodesic arcs $\mu_i:= x_{i}^{+} x_{i+1}^{-}$ meeting $\lambda_i, \lambda_{i+1}$ orthogonally at its end-points, for which 
  $$
\lim_{i\to\infty}  d({x}, \mu_i)=\infty. 
  $$
Let $D_i$ denote the length of the shortest positively oriented arc of $\lambda_i$ connecting $x_i^-$ to $x_i^+$. We let $\mu_0$ denote the shortest geodesic in $M$ 
connecting $x$ to $x_{1}^{-}$. 

\begin{figure}[H]
\centering
\includegraphics[width=7.0in]{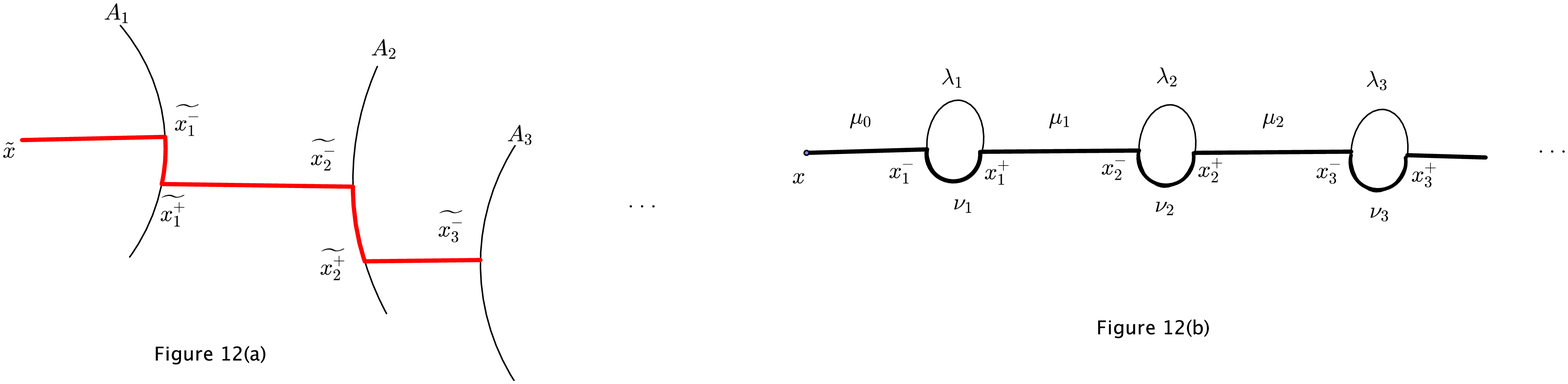} 
\caption{Here $A_i$ denotes a geodesic in $X$ covering the loop $\lambda_i$, $i\in \N$.}
\end{figure}

We next construct a family of piecewise geodesic paths $\gamma_\tau$ in $M$ starting at $x$ such that the geodesic pieces of 
$\gamma_\tau$ are the arcs $\mu_i$ above and arcs $\nu_i$ whose images are contained in $\lambda_i$ and which have the same orientation as $\lambda_i$: 
Each $\nu_i$ wraps around $\lambda_i$ a certain number of times and connects $x_i^-$ to $x_i^+$. More formally, 
 we define a map $\mathscr{P}: \mathbb{N}^{\infty}\rightarrow P(M)$ where $\mathbb{N}^{\infty}$ is the set of sequences of positive integers and $P(M)$ 
 is the space of paths in $M$ as follows:
$$\mathscr{P}: \tau=(t_{1}, t_{2}, \cdots, t_{i}, \cdots )\mapsto {\gamma}_{\tau}= \mu_0\ast \nu_{1}\ast \mu_{1}\ast \nu_{2}\ast \mu_{2}\ast \cdots \ast \nu_{i}\ast \mu_{i}\ast \cdots $$
where the image of the geodesic arc $\nu_{i}$  is contained in $\lambda_{i}$ and $\nu_i$ 
has length 
$$
l(\nu_{i})=t_{i}m_{i}l_{i}+ D_i.
$$
 Observe that for $i\ge 1$, the arc $\mu_{i}$ connects $\lambda_{i}$ and $\lambda_{i+1}$ and is orthogonal to both, with length $l(\mu_{i})\geq L$ and 
 $\nu_{i}$ starts at $x_{i}^{-}$ and ends at $x_{i}^{+}$ with length $l(\nu_{i})\geq L$.

For each ${\gamma}_{\tau}$, we have a canonical lift $\tilde\gamma_{\tau}$ in ${X}$, which is a path starting at $\tilde{x}$. We will use the notation 
$\tilde\mu_i, \tilde\nu_i$ for the lifts of the subarcs $\mu_i, \nu_i$ respectively, see Figure 12(a, b). By the construction, each $\gamma_{\tau}$ has the following properties:

\begin{enumerate}
			
			\item 
		Each geodesic piece of $\tilde\gamma_{\tau}$ has length  at least $L$. 
			\item
			Adjacent geodesic segments of $\tilde\gamma_{\tau}$ make the angle equal to $\pi /2$ at their common endpoint. 
		    
		    \item The path $\gamma_\tau: [0,\infty)\to M$ is a proper map. 
		
		\end{enumerate}
		
By Proposition \ref{piecewise geodesic path}, $\tilde\gamma_{\tau}$ is a $(2L, 4L+1)$-quasigeodesic. Hence, there exists a limit
$$
\lim_{t\to\infty} \tilde\gamma_\tau(t) = \tilde\gamma_{\tau}(\infty)\in \geo X,
$$ 
and the Hausdorff distance between $\tilde\gamma_{\tau}$ and $x\tilde\gamma_{\tau}(\infty)$ is bounded above by a uniform constant $C$, depending only on 
$L$ and $\kappa$. 

We claim that each $\tilde\gamma_{\tau}(\infty)$ is a nonconical limit point. Observe that $\tilde\gamma_{\tau}(\infty)$ is a limit of loxodromic fixed points, so 
$\tilde\gamma_{\tau}(\infty)\in \Lambda(\Gamma)$. Let $\gamma^{\ast}_{\tau}$ be the projection of $x\tilde\gamma_{\tau}(\infty)$ under $\pi$. Then the image of 
$\gamma^{\ast}_{\tau}$ is uniformly close to ${\gamma}_{\tau}$. Since ${\gamma}_{\tau}$ is a proper path in $M$, so is $\gamma^{\ast}_{\tau}$. 
Hence, $\tilde\gamma_{\tau}(\infty)$ is a nonconical limit point of $\Gamma$. 

We claim that the set of nonconical limit points $\tilde\gamma_{\tau}(\infty)$, $\tau\in {\mathbb N}^\infty$, 
has the  cardinality of the continuum. It suffices to prove that the map 
$$
\mathscr{P}_\infty: \tau \mapsto \tilde\gamma_{\tau}(\infty)$$
 is injective.  

Let $\tau=(t_{1}, t_{2}, \cdots, t_{i})$ and $\tau'=(t'_{1}, t'_{2}, \cdots, t'_{i}, \cdots)$ be two distinct sequences of positive integers. Let $m$ be the smallest positive integer such that $t_{m}\neq t'_{m}$. Then the paths $\tilde\gamma_\tau, \tilde\gamma_{\tau'}$ can be written as concatenations 
$$
\tilde\alpha_\tau \star \tilde \nu_m \ast  \tilde\beta_\tau, \quad \tilde\alpha_\tau \star \tilde\nu'_m \ast  \tilde\beta_{\tau'},
$$
where $\tilde\alpha_\tau$ is the common  initial subpath
$$
\tilde\mu_0 \ast \tilde\nu_{1}\ast \tilde\mu_{1}\ast \tilde\nu_{2}\ast \tilde\mu_{2}\ast \cdots \ast \tilde\nu_{m-1}\ast \tilde\mu_{m-1}. 
$$
The geodesic segments $\tilde\nu_m, \tilde\nu'_m$ have the form
$$
\tilde\nu_m= \tilde{x}^-_n \tilde{x}_m^+, 
$$
$$
\tilde\nu'_m= \tilde{x}^-_n \tilde{x'}_m^+. 
$$

Consider the bi-infinite piecewise geodesic path 
$$
\sigma:=\tilde{\beta}_\tau^{-1} \star \tilde{x}_n^+ \tilde{x'}_n^+ \star \tilde\beta_{\tau'}
$$
in $X$. 
Each geodesic piece of the path has length at least $L$ and adjacent geodesic  segments of the path are orthogonal to each other. By Proposition \ref{piecewise geodesic path}, 
$\sigma$ is a complete $(2L, 4L+1)$-quasigeodesic and, hence, it is backward/forward asymptotic to distinct points in $\geo X$. These points in $\geo X$ are respectively 
$\tilde\gamma_\tau(\infty)$ and $\tilde\gamma_{\tau'}(\infty)$. Hence, the map $\mathscr{P}_\infty$ is injective. We conclude that the  endpoints of the piecewise geodesic paths 
$\tilde\gamma_{\tau}$  yield a set of nonconical limit points of $\Gamma$ which has the cardinality of the continuum. 
\end{proof}

\begin{remark}
This proof  is a simplification of Bishop's argument in \cite{Bi}, since, unlike \cite{Bi}, we have orthogonality of the consecutive segments in each 
$\gamma_\tau$. 
\end{remark}

\noindent
{\bf Proof of Theorem \ref{theo 1.3}:} The implication $(1) \Rightarrow (2)$  (a generalization of Bonahon's theorem) is the main result of Section \ref{sec:Bohanon}. 
The implication $(2) \Rightarrow (3)$ is the content of Theorem \ref{Theo 5.1}. It remains to prove that $(3)\Rightarrow (1)$. If $\Gamma$ is geometrically finite, by Theorem 1.4 $\Lambda(\Gamma)$ consists of conical limit points and bounded parabolic fixed points. Since $\Gamma$ is discrete, it is at most countable; therefore, the set of fixed points of parabolic elements of $\Gamma$ is again at most countable. If $\Lambda(\Gamma)$ contains a subset of nonconical limit points of the cardinality of the continuum, we can find a point in the limit set which is neither a conical limit point nor a parabolic fixed point. It follows that $\Gamma$ is geometrically infinite. 
\qed

\vspace{5mm}

{\bf Proof of Corollary \ref{coro 1.6}:}
If $\Gamma$ is geometrically finite, by Theorem \ref{thm:gfcharacterization}, $\Lambda(\Gamma)$ consists of conical limit points and bounded parabolic fixed points. Now we prove that if $\Lambda(\Gamma)$ consists of conical limit points and parabolic fixed points, then $\Gamma$ is geometrically finite. Suppose that $\Gamma$ is geometrically infinite. By Theorem \ref{theo 1.3}, there is a set of nonconical limit points with the cardinality of the continuum. Since the set of parabolic fixed points is at most countable, there exists a limit point in $\Lambda(\Gamma)$ which is neither a conical limit point nor a parabolic fixed point. This contradicts to our assumption. Hence, $\Gamma$ is geometrically finite. 
\qed

\section{Limit sets of ends}\label{sec:ends}

We start by reviewing the notion of {\em ends} of locally path-connected, locally compact, Hausdorff topological spaces $Z$.  
We refer to \cite{KD} for a more detailed treatment. 

An {\em end} of $Z$ is the equivalence class of a sequence of connected nonempty open sets 
$$C_1\supset C_2 \supset C_3\supset \cdots $$
of $Z$, where each $C_i, i\in \N,$ is a component of 
$K_{i}^{c}=Z\setminus K_i$, and $\{ K_{i} \}_{i\in \N}$ is an increasing family of compact subsets exhausting $Z$ with 
$$K_i\subset K_{j}, \quad \hbox{~~whenever~~} i\leq j, $$
so that 
$$\bigcup_{i\in \N}K_i=Z. $$
Here two sequences $(C_i)$, $(C'_i)$ are equivalent if  each $C_i$ contains some $C'_j$ and vice-versa. The sets $C_i$ are called {\em neighborhoods} of $e$ in $Z$. A proper continuous map (a {\em ray}) 
$\rho: \R_+\to Z$ is said to be {\em asymptotic to the end $e$} if 
for every neighborhood $C_i$ of $e$, the subset $\rho^{-1}(C_i)\subset \R_+$ is unbounded.  

In this paper we will be considering ends of two classes of topological spaces: 

(1) $Z=Y=\Core(M)$, with $M=X/\Gamma$, where 
$\Gamma$ is a discrete isometry group of a Hadamard manifold $X$ of pinched negative curvature. 

(2) $Z=\noncusp_{\varepsilon}(Y)$ (with $Y$ as above), where $\varepsilon$ is less than the Margulis constant of $X$. 

\medskip 
An end $e$ of $Y=\Core(M)$ is called {\em cuspidal} or a {\em cusp} if it can be represented by a sequence $C_i$ consisting 
of projections of $\Hull(\Lambda)\cap B_i$, where $B_i$'s are nested horoballs in $X$. (As before, $\Lambda\subset \geo X$ denotes the limit set of $\Gamma$.) Equivalently, $e$ can be represented by 
a sequence $C_i$ of components of the $\varepsilon_i$-thin part $\thin_{\varepsilon_i}(Y)$ of $Y$, 
with $\lim_{i\to\infty} \varepsilon_i= 0$. 
When $\varepsilon$ is less than the {\em Margulis constant} of $X$, components of $\thin_{\varepsilon}(Y)$ which are 
neighborhoods  of $e$ are called {\em cuspidal neighborhoods} of $e$. In view of Theorem \ref{thm:gfcharacterization}, 
the group $\Gamma$ is geometrically infinite if and only if $Y$ has at least one non-cuspidal end. Equivalently, $\Ga$ is geometrically finite if and only if $Z$ is compact, equivalently, has no ends.

\medskip 
Consider a neighborhood $C$ of an end $e$ of $Z$, where $Z$ is either $Y=\Core(M)$ or is the noncuspidal part of $Y$. 
The preimage $\pi^{-1}(C)\subset \Hull(\La)$  under the 
quotient map $\pi: X\to M$ is a countable union of components $E_j$. Then $C$ is naturally 
isometric to the quotients $E_j/\Gamma_j$, where $\Gamma_j= \textup{Stab}_{\Gamma}(E_j)$ 
is the stabilizer of $E_j$ in $\Gamma$. A point 
$$
\la \in \bigcup_{j} \Lambda(\Gamma_j) \subset \La$$
 is an \emph{end-limit point} of $C$ if one (equivalently, every) geodesic ray $\beta$ in 
$\Hull(\La)$  asymptotic to $\la$ projects to a proper ray in $Y=\Core(M)$ asymptotic to $e$. 
We let $\Lambda(C)$ denote the set of end-limit points of $C$ and let $\Lambda({e})$, the {\em end-limit set of} $e$, 
denote the intersection
$$
\bigcap_{i} \Lambda(C_i) 
$$
taken over all neighborhoods $C_i$ of $e$. (It suffices to take the intersection over a sequence $(C_i)$ representing $e$.) 
Clearly, for every end $e$,  $\Lambda(e)$ is disjoint from the conical limit set of $\Gamma$. 

The main result of this section is

\begin{thm}\label{thm:ends-of-noncusp}
For every end $e$ of $Z=\noncusp_{\varepsilon}(Y)$, $\La(e)$ has the cardinality of continuum.   
\end{thm}
\proof 
The end $e$ is represented by a nested sequence $(C_i)$ of  
components of $$K^{c}_{i}= \noncusp_{\varepsilon}(Y)\setminus K_i,$$ where  $K_i=B(x, iR)$ with 
$x\in \textup{noncusp}_{\varepsilon}(Y)$ and $R$ is the same constant as in proof of Theorem \ref{theo 1.3}. 

We first claim that there exists a sequence of closed geodesics $(\la_i)$  exiting $e$, i.e. $\la_i\subset C_i$, $i\in \N$.  
We follow Bonahon's proof in \cite{Bo}.  
By Lemma \ref{lemma 3.8}, every intersection 
$$(\textup{QHull}(\Gamma \tilde{x})/ \Gamma)\cap C_i$$
is unbounded, where $\tilde{x}$ is a lift of $x$ to $X$. 

By the argument in the proof of Theorem \ref{theo 1.3}, for every $C_n$, there exists a sequence of geodesic arcs 
$(\alpha_i)\subset C_n$ such that the Hausdorff distance hd$(\alpha_i, K_n)\rightarrow \infty$ as $i\rightarrow \infty$, and there exists a sequence of piecewise geodesic loops $w_{n}\subset C_n$ exiting $e$. These  geodesic loops $w_n$ represent loxodromic isometries $\omega_n\in \Isom(X)$. Up to a subsequence, there are two possible cases:
\begin{enumerate}
\item $l(\omega_{n})\geq \epsilon>0$ for some positive constant $\epsilon$ and all $n$. 

\item $l(\omega_{n})\rightarrow 0$ as $n\rightarrow \infty$. 

\end{enumerate}

For case (1), we use the same argument as  in the proof of  Theorem \ref{theo 1.3} to construct a sequence of closed geodesics $(\lambda_{i})$ exiting $e$. 

For case (2), let ${\mathrm T}\subset \Gamma$ be the set consisting of elliptic isometries and the identity. For  $\tilde{x}\in X$, we define 
$$
d_{\Gamma}(\tilde{x})=\min_{\gamma\in \Gamma \setminus T} d({\gamma}\tilde{x}, \tilde{x}).
$$
 For  $x\in M$, set 
 $$
 r(x)=d_{\Gamma}(\tilde{x})$$
  where $\tilde{x}\in X$ is a lift of $x$. (If $\Gamma$ is torsion free, then $r(x)$ is twice of the injectivity radius at $x$.) It is 
  clear that $r$ is a continuous function on $Y$, hence, it is bounded away from zero on compact subsets of $Y$.

Thus, $r_k:=\min_{x\in K_k} r(x)>0$. By passing to a subsequence, we assume that $l(\omega_{n})<r_1/2$. Then $Mar(\omega_{n}, r_1/2)$ is nonempty and disjoint from $K_1$ for all $n$. By Proposition \ref{free homotopic2}, 
$$d(w_{n}, Mar(\omega_{n}, r_1/2))\leq D,$$
where 
$$
D= \cosh^{-1}(\sqrt{2})\lceil \log_{2} 5K \rceil+\sinh^{-1}(4/r_1)$$
 and $K$ is the same constant as in the proof of Theorem \ref{theo 1.3}. Thus, $Mar(\omega_{n}, r_1/2)\subset C_1$ for all $n$. Inductively, we 
find a subsequence $(\omega_{i_{k}})$ such that $Mar(\omega_{i_{k}}, r_{k}/2)\subset C_{k}$. The closed geodesics $w^{\ast}_{i_{k}}\subset Mar(\omega_{i_{k}}, r_{k}/2)$ are also contained in $C_{k}$.  This is the required sequence of closed geodesics $(\lambda_{i})$ exiting the end $e$.

We then continue to argue as in the proof of Theorem \ref{Theo 5.1}. Namely,  we define a family of proper piecewise-geodesic paths $\gamma_\tau$ in $Z$. Since these rays are proper and 
the sequence $(\lambda_{i})$ exits the end $e$, the paths $\gamma_\tau$ are asymptotic to the end $e$. 
Hence, the geodesic rays $\gamma^*_\tau$ are also asymptotic to $e$. 

After choosing a lift of the starting point $x$ of all piecewise  geodesic paths $\gamma_{\tau}$ in $Z$, there is a canonical choice of the lift $\bar{\gamma}_{\tau}$ of $\gamma_{\tau}$. We claim that all the endpoints $\bar{\gamma}_{\tau}(\infty)$ belong to 
$\Lambda(e)$. It suffices to prove that $\bar{\gamma}_{\tau}(\infty)\in \Lambda(C_i)$ for all $i\geq 1$. 


Since the sequence $(C_i)$ is nested, we can find a nested sequence $(E_i)$ of lifts of $C_i$ to $X$. 
Recall that $\lambda_i\subset C_i$ for every $i$. Pick a complete geodesic $A_i\subset E_i$ which is a lift of $\lambda_i$. Each loop $\lambda_i$ represents 
an element (unique up to conjugation) $\omega_i\in \Ga$. We choose $\omega_i\in \Ga$ which preserves the geodesic $A_i$. Then $\omega_i$ 
preserves $E_i$ as well and, hence, the ideal fixed points of $\omega_i$ (the ideal end-points of the geodesic $A_i$) are in the limit set of 
$\Gamma_i= \textup{Stab}_{\Gamma}(E_{i})$. By the construction, 
$\tilde\gamma_\tau(\infty)$ is the limit of the sequence of geodesics $(A_j)$. Hence, $\tilde\gamma_\tau(\infty)$ is a limit point 
of $\Gamma_i$. Since  $\gamma^{\ast}_{\tau}$ is a proper geodesic ray asymptotic to $e$, it follows that  
$\tilde\gamma_\tau(\infty)\in \Lambda(C_i)$, as required. As in the proof of Theorem \ref{theo 1.3}, the rays 
$\gamma^*_\tau$ define continuum of distinct limit points of $\La(e)$.  Hence, $\Lambda(e)$ has the cardinality of the continuum. 
\qed

\medskip 
Since $\La(e)$ is the intersection of the limit sets $\La(C)$ taken over all neighborhoods $C\subset Z=\noncusp_{\varepsilon}(Y)$, we obtain 

\begin{cor}\label{cor:ends-of-noncusp} 
For every neighborhood $C\subset Z$ of an  end $e$ of $Z$, the limit set $\La(C)$ has the cardinality of continuum. 
\end{cor}

\noindent
{\bf Proof of Corollary \ref{cor:ends}:} If a complementary component $C$ of a compact subset of $Y$ 
is Hausdorff-close to a finite union of cuspidal neighborhoods of cusps in $Y$, then $\Lambda(C)$ is a finite union of orbits of the bounded parabolic fixed points corresponding to the cusps. Suppose, therefore, that $C$ is not Hausdorff-close to a finite union of cuspidal neighborhoods of cusps in $Y$. Thus, $C$ is also a neighborhood of an end 
$e$ of $Y$ which is not a cusp. In particular, $C\cap \noncusp_{\varepsilon}(Y)$ contains an unbounded component $C'$. 
Since $\La(C')\subset \La(C)$ and $\La(C')$ has the cardinality of continuum (Corollary \ref{cor:ends-of-noncusp}), so does $\La(C)$. \qed 

\medskip 
{\bf Proof of Corollary \ref{cusp end}:} 
If $e$ is a cuspidal end of $Y$, then $\Lambda(e)$ is the orbit of the bounded parabolic fixed point corresponding to $e$ under the group $\Gamma$. Hence, $\Lambda(e)$ is countable. Suppose, therefore, that $e$ is a non-cuspidal end. As we noted above, for every neighborhood $C$ of $e$ in $Y$, the intersection $C\cap Z= \noncusp_{\varepsilon}(Y)$ contains an unbounded component $C'$. Therefore, every nested sequence $(C_i)$  representing $e$ gives rise to a nested sequence $(C_i')$ in $Z$ representing 
an end $e'$ of $Z$. Since $\La(e')$ has the cardinality of continuum (Theorem \ref{thm:ends-of-noncusp}) 
and, by the construction, $\La(e')\subset \La(e)$, it follows 
that $\La(e)$ also has the cardinality of  continuum. \qed

\end{document}